\documentclass[reqno]{amsart}

\usepackage[utf8]{inputenc}
\usepackage{amssymb}
\usepackage{graphicx}
\usepackage{latexsym}
\usepackage{stmaryrd}
\usepackage{enumitem}
\usepackage[bookmarks]{hyperref}
\usepackage{hyperref}
\usepackage{tikz}
\usepackage{tikz-cd}
\usetikzlibrary{calc,through,intersections,arrows, trees, positioning, decorations.pathmorphing, cd}
\usepackage{relsize}
\usepackage{comment}
\usepackage{xifthen}
\usepackage{mathrsfs}
\usepackage{svg}

\newcommand{\modG}{\vert G\vert_{\Theta}}

\newcommand{\toC}[2]{
	\ifthenelse{\isempty{#1}}
		{\varphi_{#2}}
		{\varphi_{#2}(#1)}
}

\newcommand{\invLim}{\varprojlim}
\newcommand{\cb}{\mathfrak{c}}

\newcommand{\rest}{\upharpoonright}
\newcommand{\rise}{\downharpoonright}

\newcommand{\subclosureIn}[2]{{\normalfont\text{cl}}{}_{#2}\,(\,#1\,)}
\DeclareMathOperator{\medcup}{\mathsmaller{\bigcup}}

\DeclareMathOperator{\Gsim}{\sim}
\DeclareMathOperator{\bGsim}{\sim_\beta}
\DeclareMathOperator{\G*sim}{\approx_\ast}
\DeclareMathOperator{\Gxsim}{\approx_\times}
\newcommand{\ulim}[1]{{#1}\text{-}{\lim}}
\newcommand{\uI}{\mathbb{I}}
\renewcommand{\subset}{\subseteq}
\renewcommand{\supset}{\supseteq}

\def\comp{com\-pac\-ti\-fi\-ca\-tion}

\def\HD{Haus\-dorff}
\def\HDcomp{\HD\ \comp}
\def\SC{Stone-Čech}

\newcommand{ \M } { \mathbb{M} }
\newcommand{ \N } { \mathbb{N} }

\newcommand{\Hstar}{\mathbb{H}^*}

\newcommand{\parentheses}[1]{{\left( {#1} \right)}}
\newcommand{\sequence}[1]{{\langle {#1} \rangle}}
\newcommand{\p}{\parentheses}
\newcommand{\of}{\parentheses}
\newcommand{\closure}[1]{\overline{#1}}
\newcommand{\closureIn}[2]{\closure{#1}^{#2}}

\newcommand{\Set}[1]{{\left\lbrace {#1} \right\rbrace}}
\newcommand{\singleton}{\Set}

\def\set#1:#2{\Set{{#1} \colon {#2}}}

\def\sequence#1:#2{\left \langle #1 \; \colon \; #2 \right \rangle}

\makeatletter

\def\calCommandfactory#1{%
   \expandafter\def\csname c#1\endcsname{\mathcal{#1}}}
\def\frakCommandfactory#1{%
   \expandafter\def\csname frak#1\endcsname{\mathfrak{#1}}}

\newcounter{ctr}
\loop
  \stepcounter{ctr}
  \edef\X{\@Alph\c@ctr}
  \expandafter\calCommandfactory\X
  \expandafter\frakCommandfactory\X
  \edef\Y{\@alph\c@ctr}
  \expandafter\frakCommandfactory\Y
\ifnum\thectr<26
\repeat

\renewcommand{\cC}{\mathscr{C}}

\renewcommand{\cK}{\mathscr{K}}
\renewcommand{\cP}{\mathscr{P}}

\newtheorem{theorem}{Theorem}[section]
\newtheorem{mainTheorem}{Theorem}
\newtheorem{proposition}[theorem]{Proposition}
\newtheorem{corollary}[theorem]{Corollary}
\newtheorem{lemma}[theorem]{Lemma}
\newtheorem{lemmaDefinition}[theorem]{Lemma and Definition}
\newtheorem{obs}[theorem]{Observation}

\newtheorem{innercustomthm}{Theorem}

\newenvironment{customthm}[1]
  {\renewcommand\theinnercustomthm{#1}\innercustomthm}
  {\endinnercustomthm}

\theoremstyle{definition}
\newtheorem{example}[theorem]{Example}
\newtheorem{fact}[theorem]{Fact}
\newtheorem{definition}[theorem]{Definition}

\theoremstyle{remark}
\newtheorem*{notation}{Notation}

\setenumerate{label={\normalfont (\roman*)},itemsep=0pt}

\newif\ifshow
\showtrue
\ifshow
	{}
\else
	\excludecomment{proof}
\fi

\def\lowfwd #1#2#3{{\mathop{\kern0pt #1}\limits^{\kern#2pt\raise.#3ex
\vbox to 0pt{\hbox{$\scriptscriptstyle\rightarrow$}\vss}}}}
\def\lowbkwd #1#2#3{{\mathop{\kern0pt #1}\limits^{\kern#2pt\raise.#3ex
\vbox to 0pt{\hbox{$\scriptscriptstyle\leftarrow$}\vss}}}}

\def\vS{{\vec S}}
\def\vSdash{{\mathop{\kern0pt S\lower-1pt\hbox{${}
     \scriptstyle'$}}\limits^{\kern2pt\raise.1ex
     \vbox to 0pt{\hbox{$\scriptscriptstyle\rightarrow$}\vss}}}}

\begin{document}

\title[Tangles and the Stone-Čech compactification of infinite graphs]{Tangles and the Stone-Čech compactification\\of infinite graphs}


\author{Jan Kurkofka}
\address{University of Hamburg, Department of Mathematics, Bundesstraße 55 (Geomatikum), 20146 Hamburg, Germany}
\email{jan.kurkofka@uni-hamburg.de, max.pitz@uni-hamburg.de}

\author{Max Pitz}

\keywords{infinite graph; Stone-Čech compactification; end; tangle}

\subjclass[2010]{54D35, 05C63}

\begin{abstract}
We show that the tangle space of a graph, which compactifies it, is a quotient of its \SC\ remainder obtained by contracting the connected components.
%
\end{abstract}
\vspace*{-2.7cm}
\maketitle

\vspace*{-.7cm}
\section{Introduction}

Every locally finite connected graph can be naturally compactified by its ends to form its well-known end \comp, see e.g.~\cite[\S 8.6]{Bible}.
For graphs that are not locally finite, however, adding their ends no longer suffices to compactify them, and it has been a longstanding quest to decide what other `points at infinity' besides the ends should be added to obtain a \comp , see e.g.~Cartwright, Soardi and Woess~\cite{Woess93} and Polat~\cite{Polat90}.

Recently, Diestel~\cite{EndsAndTangles} proposed a solution to this problem employing Robertson and Seymour's notion of a tangle~\cite{GMX}, which naturally generalises the end \comp\ (using the terminology from~\cite[§12.5]{Bible}): First, he observed that an end $\omega$ of a graph $G$ orients every finite-order separation $\{A,B\}$ of $G$ towards the side that contains a tail from every ray in $\omega$; and since these orientations for distinct separations are consistent in a number of ways, every end naturally induces an infinite-order tangle of $G$ in this way.
Diestel then proceeded to show that, conversely, every infinite-order tangle of a locally finite connected graph $G$ is defined by an end in this way.
Thus, if $G$ is locally finite and connected, there is a canonical bijection between its infinite-order tangles and its ends.

Finally, Diestel showed that \emph{every} graph, in particular also the non-locally finite ones, is compactified by its infinite-order tangles in much the same way as the ends of a locally finite connected graph compactify it in its end-compactification. The arising \emph{tangle \comp } coincides with the end \comp\ if $G$ is locally finite and connected. 
Hence, for the tangle compactification, it is precisely those infinite-order tangles not corresponding to an end which need to be added as points at infinity besides the ends in order to compactify the graph.

Diestel concludes his paper with the question of how the tangle \comp\ of an infinite graph relates to its \SC\ \comp~\cite[§6]{EndsAndTangles}.
Indeed, it is well-known that the end \comp\ of a locally finite connected graph $G$ can be described naturally in terms of its \SC\ \comp, namely, it is the quotient obtained by collapsing each connected component of the \SC\ remainder to a single point, see e.g.~\cite[\S VI.3]{aarts1993dimension}.
As our main result, we show that this correspondence extends to all graphs when ends are generalised to tangles. Hence, even though Diestel's reasoning and motivation behind the tangle \comp\ was purely combinatorial, it naturally happens to generalise the end \comp\ also in this second, more topological aspect.

\begin{mainTheorem}\label{mainresult2}
The tangle \comp\ of any graph $G$ is obtained from its \SC\ \comp\ $\beta G$ by first declaring $G$ to be open\footnote{When $G$ is locally compact, it is automatically open in $\beta G$, and so this step is redundant for locally finite graphs.} in $\beta G$ and then collapsing each connected component of the \SC\ remainder to a single point.
\end{mainTheorem}

This paper is organised as follows: First, in Section~\ref{sec2} we recall graph-theoretic background and provide a brief summary of Diestel's tangle compactification of an infinite graph. In Section~\ref{section_inversetangles}, we describe the remainder of the tangle compactification as an inverse limit of finite discrete spaces. In Section~\ref{sec4}, we provide the necessary background on the Stone-\v{C}ech compactification, and explain how the quotient relation defining the 1-complex $G$ can be used to describe the Stone-\v{C}ech compactification of an infinite graph as a `fake 1-complex' on standard intervals and \emph{non-standard} intervals (where the non-standard intervals are the standard subcontinua of the remainder of the positive half-line). Sections~\ref{sec5} and~\ref{sec6} contain the proof of our main theorem. We conclude this paper in Section~\ref{sec7} with three additional observations about the tangle compactification that might be of independent interest. In particular, we show that no compactification of a non-locally finite graph can both be Hausdorff and have a totally disconnected remainder.

\section{Reviewing Diestel's tangle compactification}
\label{sec2}

From now on, we fix an arbitrary connected simple infinite graph $G=(V,E)$.

\subsection{The 1-complex of a graph}
In the \emph{1-complex} of $G$ which we denote also by $G$, every edge $e=xy$ is a homeomorphic copy $[x,y]:=\{x\}\sqcup\mathring{e}\sqcup\{y\}$ of $\uI=[0,1]$ with $\mathring{e}$ corresponding to $(0,1)$ and points in $\mathring{e}$ being called \emph{inner edge points}. The space $[x,y]$ is called a \emph{topological edge}, but we refer to it simply as \emph{edge} and denote it by $e$ as well.
For each subcollection $F\subseteq E$ we write $\mathring{F}$ for the set $\bigsqcup_{e\in F}\mathring{e}$ of inner edge points of edges in $F$. By $E(v)$ we denote the set of edges incident with a vertex $v$.
The point set of $G$ is $V\sqcup\mathring{E}$, and an open neighbourhood basis of a vertex $v$ of $G$ is given by the unions $\bigcup_{e\in E(v)}[v,i_e)$ of half open intervals with each $i_e$ some inner edge point of $e$.
Note that the 1-complex of $G$ is (locally) compact if and only if the graph $G$ is (locally) finite, and also that the 1-complex fails to be first-countable at vertices of infinite degree.
Note that if the graph $G$ has no isolated vertices, then its 1-complex can be obtained from the disjoint sum $\bigoplus_{e\in E}\uI_e$ of copies $\uI_e$ of the unit interval by taking the quotient with respect to a suitable equivalence relation on $\bigoplus_{e\in E}\{0,1\}$.

\subsection{Combinatorial ends of graphs}
\label{SummaryEndsAndTangles}
Given a graph $G=(V,E)$ we write $\cX$ for the collection of all finite subsets of its vertex set $V$, partially ordered and directed by inclusion. A (combinatorial) \emph{end} of a graph is an equivalence class of rays, where a \emph{ray} is a 1-way infinite path, \cite{halin64}.
Two rays are \emph{equivalent} if for every $X\in\cX$ both have a subray (also called \emph{tail}) in the same component of $G-X$.
In particular, for every end $\omega$ of $G$ there is a unique component of $G-X$ in which every ray of $\omega$ has a tail, and we denote this component by $C(X,\omega)$.
Whenever we say end, we mean a combinatorial one.
The set of ends of a graph $G$ is denoted by $\Omega=\Omega(G)$. 
Further details on ends as well as any graph-theoretic notation not explained here can be found in Diestel's book~\cite{Bible}, especially in Chapter~8.

If $\omega$ is an end of $G$, then the components $C(X,\omega)$ are compatible in that they form elements of the inverse limit of the system $\{\cC_X,\cb_{X',X},\cX\}$ where $\cC_X$ is the set of components of $G-X$ and for $X' \supseteq X$, the bonding map $\cb_{X',X}\colon\cC_{X'}\to\cC_X$ sends each component of $G-X'$ to the unique component of $G-X$ including it.
Clearly, the inverse limit consists precisely of the directions of the graph: choice maps $f$ assigning to every $X\in\cX$ a component of $G-X$ such that $f(X')\subseteq f(X)$ whenever $X'\supseteq X$.
In 2010, Diestel and K\"uhn~\cite{Ends} showed that

\begin{theorem}[{\cite[Theorem 2.2]{Ends}}]\label{EndsAreDirections}
Let $G$ be any graph. Then there is a canonical bijection between the (combinatorial) ends of $G$ and its directions, i.e. $\Omega =\invLim{}\cC_X$.
\end{theorem}

\subsection{Tangles}\label{tangleSection}
Next, we formally introduce tangles for a particular type of `separation system', referring the reader to~\cite{AbstractSepSys} for an overview of the full theory and its applications.
More precisely, we introduce a definition of $\aleph_0$-tangles provided by Diestel~\cite{EndsAndTangles} which, as he proved, is equivalent to the original one due to Robertson and Seymour~\cite{GMX}.
In the next subsection, however, we explain a third, equivalent viewpoint for tangles (due to Diestel), which describes $\aleph_0$-tangles as the elements of the compact \HD\ inverse limit $\invLim{}\beta (\cC_X)$ and which we take as our point of reference for the remainder of this paper.

A (\emph{finite order}) \emph{separation} of a graph $G$ is a set $\{A,B\}$ with $A\cap B$ finite and $A\cup B=V$ such that $G$ has no edge between $A\setminus B$ and $B\setminus A$.
The collection of all finite order separations is denoted by $S$.
The ordered pairs $(A,B)$ and $(B,A)$ are then called the \emph{orientations} of the separation $\{A,B\}$, or (\emph{oriented}) \emph{separations}.
Informally we think of $A$ and $B$ as the \emph{small side} and the \emph{big side} of $(A,B)$, respectively. 
Furthermore, we think of the separation $(A,B)$ as \emph{pointing towards} its big side $B$ and \emph{away from} its small side $A$.
We write $\vS$ for the collection of all oriented separations. 
A subset $O$ of $\vS$ is an \emph{orientation} if it contains precisely one of $(A,B)$ and $(B,A)$ for each separation $\{A,B\}\in S$.

We define a partial ordering $\le$ on $\vS$ by letting
\begin{align*}
(A,B)\le (C,D):\Leftrightarrow A\subseteq C\text{ and }B\supseteq D.
\end{align*}
Here, we informally think of the oriented separation $(A,B)$ as \emph{pointing towards} $\{C,D\}$ and its orientations, whereas we think of $(C,D)$ as \emph{pointing away from} $\{A,B\}$ and its orientations.
If $O$ is an orientation and no two distinct separations $(B,A)$ and $(C,D)$ in $O$ satisfy $(A,B)<(C,D)$, i.e., no two distinct separations in $O$ point away from each other, then we call $O$ \emph{consistent}.

We call a set $\sigma\subseteq\vS$ of oriented separations a \emph{star} if every two distinct separations $(A,B)$ and $(D,C)$ in $\sigma$ point towards each other, i.e. satisfy $(A,B)\le (C,D)$.
The \emph{interior} of a star $\sigma=\{\,(A_i,B_i)\mid i\in I\,\}$ is the intersection $\bigcap_{i\in I}B_i$ of the big sides.

\begin{definition}
An \emph{$\aleph_0$-tangle} (of $G$) is a consistent orientation of $S$ that contains no finite star of finite interior as a subset.
We write $\Theta$ for the set of all $\aleph_0$-tangles.
\end{definition}

\subsection{Ends and Tangles}

If $\omega$ is an end of $G$, then letting
\begin{align*}
\tau_\omega:=\{\,(A,B)\in\vS\mid C(A\cap B,\omega)\subseteq G[B\setminus A]\,\}
\end{align*}
defines an injection $\Omega\hookrightarrow\Theta,\;\omega\mapsto\tau_\omega$ from the ends of $G$ into the $\aleph_0$-tangles. Therefore, we call the tangles of the form $\tau_\omega$ the \emph{end tangles} of $G$.
By abuse of notation we write $\Omega$ for the collection of all end tangles of $G$, so we have $\Omega\subseteq\Theta$.

In order to understand the $\aleph_0$-tangles that are not ends, Diestel studied an inverse limit description of $\Theta$ which we introduce in a moment.
First, we note that every finite order separation $\{A,B\}$ corresponds to the bipartition $\{\cC,\cC'\}$ of the component space $\cC_X$ with $X=A\cap B$ and
\begin{align*}
\{A,B\}=\big\{\,V[\cC]\cup X\,,\,X\cup V[\cC']\,\big\}
\end{align*}
where $V[\cC]=\bigcup_{C\in\cC}V(C)$,
and this correspondence is bijective for fixed $X\in\cX$. 
For all $\cC\subseteq\cC_X$ let us write $s_{X\to\cC}$ for the separation $(V\setminus V[\cC]\,,\,X\cup V[\cC])$. Hence if $\tau$ is an $\aleph_0$-tangle of the graph, then for each $X\in\cX$ it also chooses one \emph{big side} from each bipartition $\{\cC,\cC'\}$ of $\cC_X$, namely the $\cK\in\{\cC,\cC'\}$ with $s_{X\to\cK}\in\tau$. 
Since it chooses theses sides consistently, it induces an ultrafilter $U(\tau,X)$ on $\cC_X$, one for every $X\in\cX$, which is given by
\begin{align*}
U(\tau,X)=\{\,\cC\subseteq\cC_X\mid s_{X\to\cC}\in\tau\,\},
\end{align*}
and these ultrafilters are compatible in that they form a limit of the inverse system $\{\,\beta(\cC_X)\,,\,\beta(\cb_{X',X})\,,\,\cX\,\}$.
Here, each set $\cC_X$ is endowed with the discrete topology and $\beta(\cC_X)$ denotes its \SC\ \comp .
Every bonding map $\beta(\cb_{X',X})$ is the unique continuous extension of $\cb_{X',X}$ that is provided by the \SC\ property (see Theorem~\ref{thm_stonecechchar}~(ii)).
As one of his main results, Diestel showed that the map
\begin{align*}
\tau\mapsto (\,U(\tau,X)\mid X\in\cX\,)
\end{align*}
defines a bijection between the tangle set $\Theta$ and the inverse limit $\invLim{}\beta (\cC_X)$.
From now on, we view the tangle space $\Theta$ as the compact \HD\ space $\invLim{}\beta (\cC_X)$.

In his paper, Diestel moreover showed that the ends of $G$ are precisely those $\aleph_0$-tangles whose induced ultrafilters are all principal.
For every $\aleph_0$-tangle $\tau$ we write $\cX_\tau$ for the collection of all $X\in\cX$ for which the induced ultrafilter $U(\tau,X)$ is free. The set $\cX_\tau$ is empty if and only if $\tau$ is an end tangle; an $\aleph_0$-tangle $\tau$ with $\cX_\tau$ non-empty is called an \emph{ultrafilter tangle}. For every ultrafilter tangle $\tau$ the set $\cX_\tau$ has a least element $X_\tau$ of which it is the up-closure. We characterised the sets of the form $X_\tau$ combinatorially in \cite[Theorem~4.10]{EndsTanglesCrit}: they are precisely the \emph{critical} vertex sets of $G$, finite sets $X\subset V$ whose deletion leaves some infinitely many components each with neighbourhood precisely equal to $X$, and they can be used together with the ends to compactify the graph, \cite[Theorem~4.11]{EndsTanglesCrit}.

We conclude our summary of `Ends and tangles' with the formal construction of the tangle \comp .
To obtain the tangle \comp\ $\modG$ of a graph $G$ we extend the 1-complex of $G$ to a topological space $G\sqcup\Theta$ by declaring as open in addition to the open sets of $G$, for all $X\in\cX$ and all $\cC\subseteq\cC_X$, the sets
\begin{align*}
\cO_{\modG}(X,\cC):=\medcup\cC\cup\mathring{E}(X,\medcup\cC)\cup\{\,\tau\in\Theta\mid\cC\in U(\tau,X)\,\}
\end{align*}
and taking the topology this generates. Notably, $\modG$ contains $\Theta$ as a subspace.
\begin{theorem}[{\cite[Theorem 1]{EndsAndTangles}}]
\label{DiestelsTangleCompWorks}
Let $G$ be any graph, possibly disconnected.
\begin{enumerate}
\item $\modG$ is a \comp\ of $G$ with totally disconnected remainder.
\item If $G$ is locally finite and connected, then $\modG$ coincides with the Freudenthal \comp\ of $G$.
\end{enumerate}
\end{theorem}

The tangle \comp\ is \HD\ if and only if $G$ is locally finite.
However, the subspace $\modG\setminus\mathring{E}$ is compact \HD .
Teegen~\cite{MTeegen} generalised the tangle \comp\ to topological spaces.

\section{Tangles as inverse limit of finite spaces}
\label{section_inversetangles}

The \SC\ \comp\ of a discrete space can be viewed as the inverse limit of all its finite partitions, where each finite partition carries the discrete topology.
In this section, we extend this fact to the tangle space.

We start by choosing the point set for our directed poset:
\begin{align*}
\Gamma:=\{\,(X,P)\mid X\in\cX\text{ and }P\text{ is a finite partition of }\cC_X\,\}.
\end{align*} 
\begin{notation}
If an element of $\Gamma$ is introduced just as $\gamma$, then we write $X(\gamma)$ and $P(\gamma)$ for the sets satisfying $(X(\gamma),P(\gamma))=\gamma$.
Given $X\subset X'\in\cX$ and a finite partition $P$ of $\cC_X$ we write $P\downharpoonright X'$ for the finite partition
\begin{align*}
\{\,\cb_{X',X}^{-1}(\cC)\mid\cC\in P\,\}\setminus\{\emptyset\}
\end{align*}
that $P$ induces on $\cC_{X'}$.
\end{notation}
Letting $(X,P)\le (Y,Q)$ whenever $X\subset Y$ and $Q$ refines $P\downharpoonright Y$ defines a directed partial ordering on $\Gamma$:

\begin{lemma}
\label{directed}
$(\Gamma,\le)$ is a directed poset.
\end{lemma}
\begin{proof}
Checking the poset properties is straightforward; we verify that it is
directed: Given any two elements $(X,P)$ and $(Y,Q)$ of $\Gamma$ let $R$ be the coarsest refinement of $P\downharpoonright (X\cup Y)$ and $Q\downharpoonright (X\cup Y)$. 
Then $(X,P),(Y,Q)\le (X\cup Y,R)\in\Gamma$.
\end{proof}

For a reason that will become clear in the proof of our next theorem, we consider a cofinal subset of $\Gamma$, namely
\begin{align*}
\Gamma':=\{\,\gamma\in\Gamma\mid \forall\,\cC\in P(\gamma)\colon V[\cC]\text{ is infinite}\,\}.
\end{align*}

\begin{lemma}
\label{Gamma'cofinal}
$\Gamma'$ is cofinal in $\Gamma$.
\end{lemma}
\begin{proof}
Given $(X,P)\in\Gamma$ we put
\begin{align*}
X'=X\cup\medcup\{\,V[\cC]\mid\cC\in P\text{ with }V[\cC]\text{ finite}\,\}.
\end{align*}
Then $(X,P)\le (X',P\downharpoonright X')\in\Gamma'$ as desired.
\end{proof}

We aim to describe the tangle space as an inverse limit of finite \HD\ spaces.
For this, we choose $\Gamma$ as our directed poset, and for each $\gamma\in\Gamma$ we let $\cP_\gamma$ be the set $P(\gamma)$ endowed with the discrete topology.
Our bonding maps $f_{\gamma',\gamma}\colon \cP_{\gamma'}\to \cP_{\gamma}$ send each $\cC'\in\cP_{\gamma'}$ to the unique $\cC\in \cP_{\gamma}$ with $\cC'\rest X(\gamma)\subset\cC$.
Since the spaces $\cP_\gamma$ are compact \HD , so is their inverse limit
\begin{align*}
\cP:=\invLim\;(\,\cP_\gamma\mid \gamma\in\Gamma\,).
\end{align*}
By \cite[Corollary~2.5.11]{EngelkingBook} we may replace $\Gamma$ with its cofinal subset $\Gamma'$ without changing the inverse limit $\cP$, so we assume without loss of generality that $\Gamma=\Gamma'$.

\begin{notation}
If $\tau$ is an $\aleph_0$-tangle and $\gamma=(X,P)\in\Gamma$ is given, then we write $\cC(\tau,\gamma)$ for the unique partition class of $P$ that is contained in the ultrafilter $U(\tau,X)$.
\end{notation}

\begin{theorem}
\label{ThetaInvLim}
For any graph $G$, its tangle space is homeomorphic to the inverse limit $\cP$, i.e.\ $\Theta\cong\cP$.
\end{theorem}
\begin{proof}
Letting $\varphi_\gamma\colon\Theta\to\cP_\gamma$ assign $\cC(\tau,\gamma)$ to each tangle $\tau\in\Theta$ defines a collection of maps 
that are compatible as tangles are consistent. To see that our maps are continuous, it suffices to note that for all $\gamma\in\Gamma$ and $\cC\in \cP_\gamma$ we have
\begin{align*}
\varphi_\gamma^{-1}(\cC)=\{\,\tau\in\Theta\mid\cC\in U(\tau,X(\gamma))\,\}.
\end{align*}
The set $V[\cC]$ is infinite due to $\Gamma=\Gamma'$, so Diestel's \cite[Lemma 3.7]{EndsAndTangles} ensures that the preimage $\varphi_\gamma^{-1}(\cC)$ is non-empty, i.e. that our maps are surjective.
Since the tangle space $\Theta$ is compact and the inverse limit $\cP$ is \HD , the maps $\varphi_\gamma$ combine into a continuous surjection $\varphi\colon\Theta\twoheadrightarrow\cP$ (cf.~\cite[Corollary~3.2.16]{EngelkingBook}).
Moreover, $\varphi$ is injective, so it follows from compactness that $\varphi$ is a homeomorphism.
\end{proof}

\section{Background on the Stone-\texorpdfstring{\v{C}}{C}ech compactification of an infinite graph}

\label{sec4}

\subsection{Stone-\texorpdfstring{\v{C}}{C}ech compactification of 1-complexes}
\label{sec_beta1}

The following characterisation of the Stone-\v{C}ech compactification is well-known:
\begin{theorem}[Cf.\ {\cite{EngelkingBook},\cite{GillmanJerison}}]
\label{thm_stonecechchar}
Let $X$ be a Tychonoff space. The following are equivalent for a Hausdorff compactification $\gamma X \supset X$:
\begin{enumerate}
\item $\gamma X = \beta X$,
\item every continuous function $f \colon X \to T$ to a compact Hausdorff space $T$ has a continuous extension $\hat{f} \colon \gamma X \to T$ with $\hat{f} \restriction X = f$,
\item every continuous function $f \colon X \to \uI$ has a continuous extension $\hat{f} \colon \gamma X \to \uI$ with $\hat{f} \restriction X = f$.
\end{enumerate}
Moreover, if $X$ is normal\footnote{In this paper, the property \emph{normal} always includes \emph{Hausdorff}.}, then we may add
\begin{enumerate}[resume]
\item any two closed disjoint sets $Z_1 ,Z_2 \subset X$ have disjoint closures in $\gamma X$,
\item for any two closed sets $Z_1, Z_2 \subset X$ we have
\[\closureIn{Z_1 \cap Z_2}{\gamma G} = \closureIn{Z_1}{\gamma G} \cap\closureIn{Z_2}{\gamma G}. \qedhere\]
\end{enumerate}
\end{theorem}
%

Remarkably, from (iv) it follows that whenever $X$ is normal and $Y\subset X$ is closed, then $\closureIn{Y}{\beta X}=\beta Y$. (Also cf.~\cite[Corollary~3.6.8]{EngelkingBook}.) In particular $\closureIn{V}{\beta G}=\beta V$.


\subsubsection*{Ultrafilter limits}
Consider a compact \HD\ space $X$.
If $x=(\,x_i\mid i\in I\,)$ is a family of points $x_i\in X$ and $U$ is an ultrafilter on the index set $I$, then there is a unique point $x_U\in \closure{\{\,x_i\mid i\in I\}}\subset X$ defined by
\begin{align*}
\{x_U\}=\bigcap_{J\in U}\closure{\{\,x_i\mid i\in J\,\}}.
\end{align*}
Indeed, since $U$ is a filter, the collection $\big\{\,\{x_i\colon i\in J\}\;\big\vert\;J\in U\,\big\}$ has the finite intersection property, and so by compactness of $X$, the intersection over their closures is non-empty; and it follows from Hausdorffness of $X$ that the intersection can contain at most one point.
We also write
\begin{align*}
x_U=\ulim{U}\,x=\ulim{U}\,(\,x_i\mid i\in I\,)
\end{align*}
and call $x_U$ the limit of $(\,x_i\mid i\in I\,)$ along $U$, or $U$-limit of $x$.
Note that if $U$ is the principal ultrafilter generated by $i\in I$, then $x_U=x_i$.

For an alternative description, put $T=\closure{\{\,x_i\mid i\in I\,\}} \subseteq X$ and view $I$ as a discrete space, so that the index function
\begin{align*}
\tilde{x}\colon I\to \{\,x_i\mid i\in I\,\}\subset T,\;i\mapsto x_i
\end{align*}
is continuous and $\beta I$ is given by the space of ultrafilters on $I$.
Then the \SC\ extension $\beta \tilde{x}\colon\beta I\to T$ of the index function $\tilde{x}$ maps each ultrafilter $U\in\beta I$ to $x_U$.

More generally, if $(\,X_i\mid i\in I\,)$ is a family of subsets of a compact \HD\ space $X$ and $U$ is an ultrafilter on the index set $I$, then we write
\begin{align*}
X_U=\ulim{U}\,(\,X_i\mid i\in I\,):=\bigcap_{J\in U}\closure{\bigcup_{i\in J}X_i}\subset X
\end{align*}
and call $X_U$ the $U$-limit of $(\,X_i\mid i\in I\,)$. 
Regarding ultrafilter limits, we have the following well-known lemma.

\begin{lemma}
\label{fibrelemma}
Suppose that $f\colon X\to D$ is a continuous surjection with $X$ normal and $D$ discrete. 
Then the fibres of $\beta f\colon\beta X\to\beta D$ are precisely the sets $\ulim{U}\,(\,f^{-1}(d)\mid d\in D\,)$ with $U$ an ultrafilter on $D$.
\end{lemma}
\begin{proof}
First, for an arbitrary subset $J \subset D$ the preimages $f^{-1}(J)$ and $f^{-1}(D \setminus J)$ partition $X$ into closed subsets, and hence induce a partition of $\beta X$ into closed subsets $\closure{f^{-1}(J)}$ and $\closure{f^{-1}(D \setminus J)}$. 
Since also $(\beta f)^{-1}(\closure{J})$ and $(\beta f)^{-1}(\closure{D \setminus J})$ partition $\beta X$, it follows from $\closure{f^{-1}(J)} \subseteq (\beta f)^{-1}(\closure{J})$ that $\closure{f^{-1}(J)} = (\beta f)^{-1}(\closure{J})$ for all $J \subset D$.

Therefore, for an arbitrary ultrafilter $U \in \beta D$ we have 
$$(\beta f)^{-1}(U) = (\beta f)^{-1}\bigg( \bigcap_{J \in U} \closureIn{J}{\beta D}\bigg) =  \bigcap_{J \in U} \closureIn{f^{-1}(J)}{\beta X} = \ulim{U}\,(\,f^{-1}(d)\mid d\in D\,),$$
which is the assertion of the lemma.
\end{proof}

\subsubsection*{Two facts about continua}
We shall need the following two simple lemmas about continua. Recall that a continuum is a non-empty compact connected Hausdorff space. 

\begin{lemma}
\label{closureconnd}
Let $X$ be a compact Hausdorff space, and $C \subset X$ a connected subspace. Then $\closure{C} \subset X$ is a continuum. \qed
\end{lemma}

A family $(\,C_i\mid i\in I\,)$ of subcontinua of some topological space is said to be \emph{directed} if for any $i,j\in I$ there exists a $k\in I$ such that $C_k \subset C_i \cap C_j$.

\begin{lemma}[{\cite[Theorem~6.1.18]{EngelkingBook}}]
\label{continuaintersection}
The intersection of any directed family of continua is again a continuum. 
\qed
\end{lemma}

\subsubsection*{The Stone-\v{C}ech compactification of a disjoint sum of intervals} 
Recall that the 1-complex of a connected graph $G$ can be obtained from the topological sum of disjoint unit intervals (one for each edge) by identifying suitable endpoints, and using the quotient topology. 
To formalise this, consider the topological space $\M_E = \uI \times E$ where $E = E(G)$ carries the discrete topology.
Then $G = \M_E / {\Gsim}$ for some suitable equivalence relation identifying endpoints. 
Write $\uI_e$ for $\uI \times \singleton{e} \subset \M_E$, and $x_e$ for $(x,e) \in \uI_e$, so $\M_E=\bigoplus_{e\in E}\uI_e$.

Our next results, and in particular Theorem~\ref{lem_stonecechquotient}, say that the Stone-\v{C}ech compactification of a 1-complex $G$ (which to our knowledge hasn't been studied at all) can be understood through the Stone-\v{C}ech compactification $\beta \M_E$ of $\M_E$ (which has been studied extensively over the past decades, see e.g.\ the survey \cite{hart}).

\begin{lemma}[{\cite[Corollary~2.2]{hart}}]
\label{componentsdescription}
Let $X=\bigoplus_{i\in I} K_i$ be a topological sum of continua, and view $I$ as a discrete space. Consider the continuous projection $\pi \colon X \to I$, sending $K_i$ to $i \in I$. The components of $\beta X$ are the fibres of the map $\beta \pi \colon \beta X \to \beta I$.
\end{lemma}

Suppose for a moment that $X=\bigoplus_{i\in I} K_i$ has only countably many components, i.e.\ that $I = \N$. Write $X^* = \beta X \setminus X$ for the \SC\ remainder. In the lemma, $\beta \pi$ denotes the Stone-\v{C}ech extension of $\pi$, where we interpret $\pi$ as a continuous map from $X$ into the compact Hausdorff space $\beta \N \supset \N$. And since $\pi$ has compact fibres (also called \emph{perfect map}), the extension $\beta \pi$ restricts to a continuous map $\pi^*= \beta \pi \restriction X^* \colon X^* \to \N^*$, i.e.\ it maps the remainder of $\beta X$ to the remainder of $\beta\N$, \cite[Theorem~3.7.16]{EngelkingBook}. The figure below illustrates this for $X=\M_{\N}$:

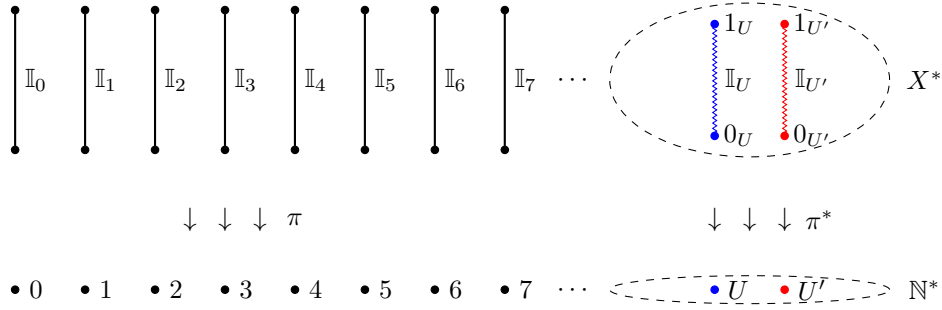
\begin{figure}[ht]
\begin{tikzpicture}[scale=0.93]
\foreach \s in {0,...,7}
{
 \node at (\s+.3,3) {$\uI_{\s}$};
 \node at (\s+.3,0) {$\s$};
  \node[draw, circle,scale=.3, fill] (U\s) at (\s,4) {};
  \node[draw, circle,scale=.3, fill] (V\s) at (\s,2) {};
  \node[draw, circle,scale=.3, fill] (W\s) at (\s,0) {};

\draw[thick] (U\s) -- (V\s);
}

\node at (2.5,1) {$\downarrow$};
\node at (3,1) {$\downarrow$};
\node at (3.5,1) {$\downarrow$};
\node at (4,1) {$\pi$};

\node at (8,3) {$\cdots$};
\node at (8,0) {$\cdots$};

\draw[dashed] (10.5,0) ellipse (2 and .2);
\node at (13,0) {$\N^*$};

\node at (11,1) {$\downarrow$};
\node at (10,1) {$\downarrow$};
\node at (10.5,1) {$\downarrow$};
\node at (11.5,1) {$\pi^*$};

\draw[dashed] (10.5,3) ellipse (2 and 1.1);
\node at (13,3) {$X^*$};

\node[draw, circle,scale=.3, fill, blue, label=right:$U$] at (10,0) {};
\node[draw, circle,scale=.3, fill, red, label=right:$U'$] at (11,0) {};

\draw[decorate, blue, decoration={zigzag, segment length=+2pt, amplitude=+.75pt}] (10,3.8) -- (10,2.2);
 \node[draw, blue, circle,scale=.3, fill] (bla) at (10,3.8) {};
 \node at (10.35,3.8) {$1_U$};
  \node[draw, blue, circle,scale=.3, fill] (bla) at (10,2.2) {};
 \node at (10.35,2.2) {$0_U$};

\draw[decorate, red, decoration={zigzag, segment length=+2pt, amplitude=+.75pt}] (11,3.8) -- (11,2.2);
	 \node[draw, red, circle,scale=.3, fill] (bla) at (11,3.8) {};
 \node at (11.4,3.8) {$1_{U'}$};
 \node[draw, red, circle,scale=.3, fill] (bla) at (11,2.2) {};
 \node at (11.4,2.2) {$0_{U'}$};
 
\node at (10.35,3) {$\uI_U$};
\node at (11.4,3) {$\uI_{U'}$};

\end{tikzpicture}
\end{figure}

Now, for every ultrafilter $U \in \beta \N$ the fibre $\beta \pi^{-1}(U)$ is a connected component of $\beta X$, which is also denoted by $K_U$. 
By Lemma~\ref{fibrelemma} we have
\begin{align*}
\beta \pi^{-1}(U) = K_U = \ulim{U}\,(\,K_i\mid i\in I\,) = \bigcap_{J \in U} \closureIn{\bigcup_{i \in J} K_i}{\beta X}.
\end{align*}
Also, if $(\,x_i\mid i\in I\,)$ is a family of points with $x_i\in K_i$, then $x_U$ is the unique point of $K_U\cap \closureIn{\{\,x_i\mid i\in I\,\}}{\beta X}$. If the spaces $K_i$ are homeomorphic copies of a single space and the points $x_i\in K_i$ correspond to the same point $\xi$ of the original space, then we write $\xi_U$ for $x_U$. For example, if each $K_i$ is a copy of the unit interval and $x_i$ corresponds to $0$ for all $i\in I$, then $x_U=0_U$.

We shall also need the following lemma plus corollary:

\begin{lemma}[{\cite[Lemma~2.3]{hart}}]
\label{cutpoints}
For a family $(\,x_i\mid i\in I\,)$ of points $x_i\in K_i$, the point $x_U$ is a cut-point of $K_U$ if and only if $\{\,i\mid x_i \textnormal{ is a cut-point of } K_i\,\} \in U$.
\end{lemma}

\begin{notation}
In the context of $X=\M_E$ we write $\check{\uI}_U$ for $\uI_U\setminus\{0_U,1_U\}$.
\end{notation}

\begin{corollary}
\label{openIntervalConnected}
The spaces $\uI_U\setminus \{0_U\}$, $\uI_U\setminus\{1_U\}$ and $\check{\uI}_U$ are connected.
\end{corollary}
\begin{proof}
The non-standard interval $[0_U,(\frac{1}{2})_U]$ is homeomorphic to $\uI_U$ (cf.~\cite[Proposition~2.8]{hart}). Thus $(0_U,(\frac{1}{2})_U]$ is connected by Lemma~\ref{cutpoints}. So is $[(\frac{1}{2})_U,1_U)$. Since both meet in $(\frac{1}{2})_U$, so is their union $\check{\uI}_U$.
\end{proof}

\subsubsection*{Quotients}
As we are interested in 1-complexes, i.e.\ in quotients of $\M_E$, we provide a theorem how the quotient operation relates to the Stone-\v{C}ech functor. We need the following lemma, which is easily verified (alternatively see Theorems~2.4.13 and~1.5.20 from~\cite{EngelkingBook}).

\begin{lemma}
\label{niceQuotientOfNormalSpaceAgainNormal}
Let $V$ be a closed discrete subset of a normal space $X$, and suppose that ${\sim}$ is an equivalence relation on $V$. Then $X/{\sim}$ is again normal.\qed
\end{lemma}
%

\begin{theorem}
\label{lem_stonecechquotient}
Let $V$ be a closed discrete subset of a normal space $X$,  and suppose that ${\Gsim}$ is an equivalence relation on $V$. 
Let $\{\,V_i\mid i\in I\,\}$ be the collection of all 
${\Gsim}$-classes. 
Consider the equivalence relation ${\bGsim}$ on $\closureIn{V}{\beta X}$ into equivalence classes of the form
\[V_U = \ulim{U}  \,(\,V_i\mid i\in I\,) = \bigcap_{J \in U} \closureIn{\bigcup_{i \in J} V_i}{\beta X},\]
one for each ultrafilter $U$ on $I$, and singletons. Then $X/{\Gsim}$ is again normal and
\[ \beta \p{X / {\sim}} = \p{\beta X} / {\sim}_\beta\;.\]
\end{theorem}

\begin{proof}
Let us write $V / {\sim} = I$ where $I$ is endowed with the discrete topology.
The quotient $X/{\sim}$ is normal by Lemma~\ref{niceQuotientOfNormalSpaceAgainNormal}, so its \SC\ \comp\ exists. Also, the quotient map $q \colon X \to X / {\sim}$ is a continuous closed map (\cite[Proposition~2.4.3]{EngelkingBook}), and so $q[V] = I$ is closed in $X/{\sim}$.

Now by the \SC\ property in Theorem~\ref{thm_stonecechchar}~(ii), the map $q \colon X \to X/{\sim} \subset \beta (X /{\sim})$ extends to a continuous surjection $\beta q \colon \beta X \to \beta (X /{\sim})$.

We claim that each non-trivial fibre of $\beta q$ is of the form $V_U$ for each ultrafilter $U$ on $I$. Since every continuous surjection $f\colon Z \twoheadrightarrow Y$ from a compact space $Z$ onto a Hausdorff space $Y$ gives rise to a homeomorphism between the quotient $Z  / \{\,f^{-1}(y)\mid y \in Y\,\}$ over the fibres of $f$ and $Y$, this implies the desired result.

First, note that $\beta q$ maps $\closureIn{V}{\beta X}$ onto $\closureIn{I}{\beta (X /\sim)}$, and restricts to a bijection on the respective complements, as $V$ is closed. Moreover, as 
$I \subset X/{\sim}$ is closed and discrete, we have 
$\closureIn{I}{\beta (X /\sim)} = \beta I = \{U \colon U \text{ is an ultrafilter on } I \}$. 
Hence, by Lemma~\ref{fibrelemma} the fibres of $\beta q$ are just $(\beta q)^{-1}(U) = V_U$, one for each ultrafilter $U$ on~$I$.
\end{proof}

\begin{corollary}
\label{betaquotient}
Let $X$ be a normal space and $V \subset X$ a closed discrete subset.
Then $X/V$ is again normal and
\[
\pushQED{\qed} 
\beta \p{X / V} = \beta X / \big(\,\closureIn{V}{\beta X}\big).\qedhere
\popQED
\]
\end{corollary}

\begin{corollary}
\label{betaquotient2}
Let $X$ and $Y$ be two disjoint normal spaces, and suppose that $A=\{\,a_i\mid i\in I\,\} \subset X$ and $B = \{\,b_i\mid i\in I\,\} \subset Y$ are infinite closed discrete subspaces.
Consider the quotient $Z=(X\oplus Y)/{\sim}$ where we identify pairs $\Set{a_i,b_i}$ for all $i\in I$. Then \[\beta Z = (\beta X \oplus \beta Y) / {\sim}_\beta\] where we identify pairs $\Set{a_U,b_U}$ for all ultrafilters $U$ on $I$.\qed
\end{corollary}

\subsection{Three examples}
\label{sec_beta2}

Before turning towards the proof of our main result, we illustrate the above topological lemmas by three representative examples: We discuss the Stone-\v{C}ech compactification of the infinite ray $R$, the infinite star $S_\lambda$ of degree $\lambda$, and the dominated ray $D$, and compare it side by side with the $\aleph_0$ tangles of these examples.

\subsubsection*{The infinite ray} Consider the infinite ray $R$ with vertex set $V=\{\,v_n\mid n \in \N\,\}$ and edge set $E=\{\,v_nv_{n+1}\mid n \in \N\,\}$. 
Since $R$ is locally finite, the space of $\aleph_0$-tangles consists solely of the single end of $R$, by Theorem~\ref{DiestelsTangleCompWorks}~(ii). Moreover, the 1-complex $R$ is homeomorphic to the positive half line $\mathbb{H}=[0,\infty)$, so they have the same Stone-\v{C}ech remainder $R^* = \Hstar$. The space $\Hstar$ has been extensively investigated, see e.g.\ \cite{hart} for a survey. At this point, however, we are content to provide the standard argument showing that the Stone-\v{C}ech remainder of the infinite ray is indeed connected, confirming the connection between components in the remainder of the Stone-\v{C}ech compactification and the $\aleph_0$-tangles.

\begin{example}
\label{ex_1}
The infinite ray has a connected Stone-\v{C}ech remainder.
\end{example}

\begin{proof}
Deleting a vertex $v_n$ from $R$ leaves behind exactly one infinite component $C_n=R[v_{n+1},v_{n+2},\ldots]$. Then
$ \bigcap_{n \in \N} \closureIn{C_n}{\beta R}$
is a continuum by Lemmas~\ref{closureconnd} and \ref{continuaintersection}. We claim that \[R^* = \bigcap_{n \in \N} \closureIn{C_n}{\beta R}.\] Indeed, ``$\supseteq$" holds as any vertex and edge of $R$ is removed eventually by the intersection. For ``$\subseteq$" note that for any $n \in \N$ we have $R = R[v_0,\ldots,v_{n+1}] \cup C_n$, and hence 
\[R^* \subset \closureIn{R[v_0,\ldots,v_{n+1}]}{\beta R} \cup \closureIn{C_n}{\beta R},\] since the closure operator distributes over finite unions. But $R[v_0,\ldots,v_{n+1}]$ is compact, and hence closed in the Hausdorff space $\beta R$, implying 
\[\closureIn{R[v_0,\ldots,v_{n+1}]}{\beta R} = R[v_0,\ldots,v_{n+1}] \subset R.\]
It follows $R^* \subset \closureIn{C_n}{\beta R}$ for all $n \in \N$ as desired.
\end{proof}

\subsubsection*{The infinite star}

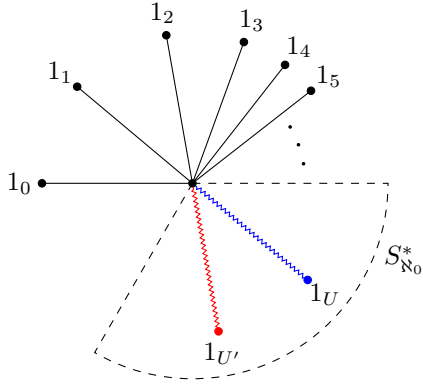
\begin{figure}[ht]
    \begin{tikzpicture}
\def \n {8}
\def \radius {2}
    \def \radiuss {2.3}
\def \margin {6} 

\node[draw, circle,scale=.3, fill] (center) at (0,0) {};
      
\node[draw, circle,scale=.3, fill] (N1) at (180:\radius) {$$};
\node[draw, circle,scale=.3, fill] (N2) at (140:\radius) {$$};
\node[draw, circle,scale=.3, fill] (N3) at (100:\radius) {$$};
\node[draw, circle,scale=.3, fill] (N4) at (70:\radius) {$$};
\node[draw, circle,scale=.3, fill] (N5) at (52:\radius) {$$};
\node[draw, circle,scale=.3, fill] (N6) at (38:\radius) {$$};

\node at (180:\radiuss) {$1_{0}$};
\node at (140:\radiuss) {$1_{1}$};
\node at (100:\radiuss) {$1_{2}$};
\node at (70:\radiuss) {$1_{3}$};
\node at (52:\radiuss) {$1_{4}$};
\node at (38:\radiuss) {$1_{5}$};

\node[draw, circle,scale=.1, fill] (dot) at (30:1.5) {$$};
\node[draw, circle,scale=.1, fill] (dot) at (20:1.5) {$$};
\node[draw, circle,scale=.1, fill] (dot) at (10:1.5) {$$};

\foreach \s in {1,...,6}
{
\draw (center) -- (N\s);
}

\draw[dashed] (0,0) -- (2.6,0) arc (0:-120:2.6) -- (0,0);

\node[blue, draw, circle,scale=.3, fill] (U1) at (-40:\radius) {$$};
\node[red, draw, circle,scale=.3, fill] (U2) at (-80:\radius) {$$};

\node at (-40:\radiuss) {$1_U$};
\node at (-80:\radiuss) {$1_{U'}$};

\draw[decorate, blue, decoration={zigzag, segment length=+2pt, amplitude=+.75pt}] (U1) -- (center);
\draw[decorate, red, decoration={zigzag, segment length=+2pt, amplitude=+.75pt}] (U2) -- (center);

\node at(-20:3) {$S_{\aleph_0}^*$};

\end{tikzpicture}
\caption{The Stone-\v{C}ech compactification of the countable infinite star}
\label{fig:infiniteStar}
\end{figure}

For any cardinal $\lambda$ we denote by $S_\lambda$ the star of degree $\lambda$. 
Clearly, this star has no end, so all $\aleph_0$-tangles are ultrafilter tangles.
As a consequence of \cite[Theorem~4.10]{EndsTanglesCrit}, the ultrafilter tangles correspond precisely to the free ultrafilters on~$\lambda$.
The 1-complex of $S_\lambda$ is obtained from $\M_E$ (with $E$ a discrete space of cardinality $\lambda$) via
\[S_\lambda = \M_E / \{\,0_e\mid e \in E\}.\]

\begin{example}
\label{infinitestar}
The Stone-\v{C}ech remainder of an infinite star $S_\lambda$ is homeomorphic to $\M_E^* \setminus \{\,0_U\mid U\in E^*\}$. 
Each connected component of $S_\lambda^\ast$ is homeomorphic to $\uI_U \setminus \singleton{0_U}$ for some free ultrafilter $U \in E^*$. 
\end{example}

\begin{proof}
Since $S_\lambda = \M_E / \{\,0_e\mid e \in E\,\}$, it follows immediately from Corollary~\ref{betaquotient} that $\beta S_\lambda = \beta \M_E / \closureIn{\{\,0_e\mid e \in E\,\}}{\beta\M_E}$. Since the equivalence class $\closureIn{\{\,0_e\mid e \in E\,\}}{\beta\M_E}$ corresponds to the center vertex of $S_\lambda$, it follows for the remainder of $\beta S_\lambda$ that
\[S^*_\lambda = \M_E^* \setminus \closureIn{\{\,0_e\mid e \in E\,\}}{\beta\M_E} =  \M_E^* \setminus \{\,0_U\mid U \in E^\ast\,\}.\]
By Lemma~\ref{componentsdescription} and Corollary~\ref{openIntervalConnected}, the connected components of the remainder $\M_E^* \setminus \{\,0_U\mid U \in E^\ast\,\}$ are given by $\uI_U \setminus \singleton{0_U}$ for each free ultrafilter $U$ on $E$.
\end{proof}

\subsubsection*{The dominated ray} 

The dominated ray $D$ is the quotient of an infinite star $S_{\aleph_0}$ and a ray $R$ where the leaves of $S_{\aleph_0}$, denoted as in the previous example by $\{\,1_n\mid n \in \N\,\}$, are identified pairwise with vertices of the ray, denoted by $\{\,v_n\mid n \in \N\,\}$ (see Fig.~\ref{fig:dominatedRay}).
Since deleting any finite set of vertices from $D$ leaves only one infinite component, the sole end of $D$ is the one and only $\aleph_0$-tangle.

\begin{figure}[ht]
\begin{tikzpicture}

 \node[draw, circle,scale=.5, fill] (Center) at (0,2) {};

\foreach \s in {0,...,7}
{
\node[draw, circle,scale=.5, fill] (U\s) at (\s,0) {};
\draw[thick] (U\s) -- (Center);
}

\foreach \s in {0,...,7}
{
pgfmathparse{\s+1}
\xdef\t{\pgfmathresult}
\draw[thick] (U\s) -- (U\t);
 \node[] () at (\s,-.3) {$v_\s$};
}
\draw[thick, ->] (U7) -- (8,0);

\node[] () at (6,1) {$\ldots$};
\node[] () at (0,2.3) {$c$};

\end{tikzpicture}
\caption{The dominated ray with dominating vertex $c$}
\label{fig:dominatedRay}
\end{figure}
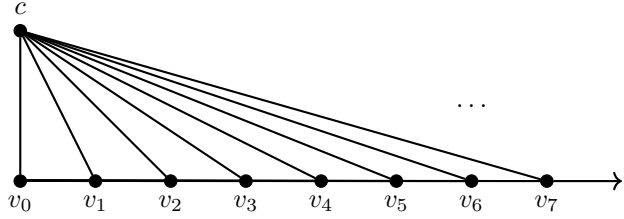

\begin{example}
\label{domray}
The dominated ray $D$ has a connected Stone-\v{C}ech remainder.
\end{example}

\begin{proof}
By Corollary~\ref{betaquotient2}, the Stone-\v{C}ech remainder of $D$ is homeomorphic to the quotient $\p{S^*_{\aleph_0} \oplus R^*}/ {\sim}_\beta$ where $1_U \bGsim v_U$ for every ultrafilter $U \in \N^\ast$ and $1_U \in \uI_U$ and $v_U \in R^*$. 
It follows that every connected component $\uI_U \setminus \Set{0_U}$ of $S^*_{\aleph_0}$ (see Example~\ref{infinitestar}) is, via the identified points $1_U \bGsim v_U$, attached to the connected remainder $R^*$ (see Example~\ref{ex_1}) of $\beta R$, and so $D^\ast$ is indeed connected.
\end{proof}

\section{Comparing the Stone-\texorpdfstring{\v{C}}{C}ech remainder with the tangle space}
\label{sec5}

\subsection{The Stone-\texorpdfstring{\v{C}}{C}ech remainder of the vertex set}

Due to $\beta G=(\beta \M_E)/{\bGsim}$ for any representation $\M_E/{\Gsim}$ of $G$ (Theorem~\ref{lem_stonecechquotient}) we may view $\beta V=\closureIn{V}{\beta G}\subset\beta G$ as the closure of $\{\,[0_e]_{\bGsim}\,,\,[1_e]_{\bGsim}\mid e\in E\,\}$ in the quotient $(\beta\M_E)/{\bGsim}$.
In particular, the non-standard intervals $\uI_U$ (with $U\in E^\ast$) may interact with $V$ or its \SC\ remainder $V^\ast$.
In this subsection, we have a closer look at this interaction.

In the next lemma, we write $V^* = G^* \cap \closureIn{V}{\beta G}$. Since $\beta V = \closureIn{V}{\beta G}$, this potential double meaning does no harm.

\begin{lemma}
\label{noTwoEndpointsSameClass}
Let $\M_E/{\Gsim}$ be a representation of $G$, and let $U\in E^\ast$ be any free ultrafilter.
Then at most one of the endpoints $0_U$ and $1_U$ of $\uI_U$ is contained in some ${\bGsim}$-class that belongs to $V$, and at least one of them is contained in some ${\bGsim}$-class that belongs to $V^\ast$.
\end{lemma}
\begin{proof}
A vertex $x\in G$, viewed as ${\bGsim}$-class (Theorem~\ref{lem_stonecechquotient}), contains an endpoint of $\uI_U$ if and only if $E(x)\in U$.
And since $\vert E(x)\cap E(y)\vert\le 1$ for every distinct two vertices $x,y\in G$, at most one vertex $x\in G$ can satisfy $E(x)$.
\end{proof}


\begin{lemma}
\label{lem_componentsandgraphs}
Let $G$ be a graph, and let $C$ be a connected component of the \SC\ remainder $G^*$. Then $C \cap V^* \neq \emptyset$. In particular, the connected components of $G^*$ induce a closed partition of $V^*$.
\end{lemma}

\begin{proof}
Consider a representation $G = \M_E / {\sim}$ of $G$, and recall that by Corollary~\ref{openIntervalConnected}, every non-standard component $\uI_U$ of $\M^*_E$ remains connected upon deleting one or both of the endpoints $0_U$ and $1_U$.

Consider some connected component $C$ of $G^*$. 
Then for some $\uI_U \subset \M^*_E$ we have $\check{\uI}_U\subset C$. Therefore, it suffices to show that for every free ultrafilter $U\in E^\ast$ at least one of $[0_U]_{\bGsim}$ and $[1_U]_{\bGsim}$ is in $V^*$.
This is the content of Lemma~\ref{noTwoEndpointsSameClass}.
\end{proof}

\subsection{An auxiliary remainder}
\label{sec_beta4}
The remainder $G^\ast$ not being compact prevents us from using topological machinery, so we study a nice subspace $G^\times\subset G^\ast$ first.
As usual, we start with some new notation.

\begin{notation}
For a vertex $v$ of $G$, write $O(v)$ for its open neighbourhood $\mathring{E}(v)\sqcup\{v\}$ in $G$ consisting of all half-open incident edges at $v$, and write
\begin{align*}
O_{\beta G}(v):=\closureIn{\medcup E(v)}{\beta G}\setminus\closureIn{N(v)}{\beta G}.
\end{align*}
\end{notation}
Due to $\beta G=\closureIn{\bigcup E(v)}{\beta G}\cup\closureIn{G\setminus O(v)}{\beta G}$ and $\closureIn{\bigcup E(v)}{\beta G}\cap\closureIn{G\setminus O(v)}{\beta G}=\closureIn{N(v)}{\beta G}$ the set $O_{\beta G}(v)$ is open in $\beta G$, and it meets $G$ precisely in $O(v)$. 
The set $O_{\beta G}(v)$ is also known as $\text{Ex }O(v)=\beta G\setminus\closure{ G\setminus O(v)}$, the largest open subset of $\beta G$ whose intersection with $G$ is $O(v)$, cf.~\cite[p.~388]{EngelkingBook}.

\begin{obs}\label{ObGvIsStar}
Put $F=E(v)$ and write $H$ for the subspace $\bigcup F\subset G$. 
Since $H$ is the 1-complex of a star, the set $O_{\beta G}(v)$ is homeomorphic to the space from Example~\ref{infinitestar} without the ``endpoints'' (also see Fig.~\ref{fig:infiniteStar}):
\begin{align*}
O_{\beta G}(v)&=\closureIn{H}{\beta G}\setminus\closureIn{N(v)}{\beta G}\cong\beta H\setminus\closureIn{N(v)}{\beta H}\\
&\cong \p{ \beta \M_F/\{\,0_U \mid U\in \beta F\,\} } \setminus\{\,1_U\mid U\in\beta F\,\}
\end{align*}
\end{obs}

\begin{definition}
\label{def_Gcross}
The auxiliary remainder of $G$ is the space
\begin{align*}
G^\times :=\beta G\setminus O_{\beta G}[V]\subset G^\ast
\end{align*}
where we write $O_{\beta G}[W]=\bigcup_{v\in W}O_{\beta G} (v)$ for all $W\subset V$.
\end{definition}

\begin{fact}
Since $\beta G$ is compact Hausdorff, so is $G^\times$.
\end{fact}

\begin{lemma}
\label{infiniteCmeetsGtimes}
The vertex set $V$ of any graph satisfies $V^* \subset G^\times$.
\end{lemma}

\begin{proof}
We show that, for every vertex $v\in V$, the set $O_{\beta G}(v)$ avoids $V^\ast$:
\begin{align*}
\closureIn{\medcup E(v)}{\beta G}\cap V^\ast &= \p{ \closureIn{\medcup E(v)}{\beta G}\cap \closureIn{V}{\beta G} }\setminus G = \closureIn{\{v\}\sqcup N(v)}{\beta G}\setminus G\\
&= \p{ \{v\}\sqcup\closureIn{N(v)}{\beta G} }\setminus G=N(v)^\ast\subset\closureIn{N(v)}{\beta G}\qedhere
\end{align*}
\end{proof}

\subsection{The components of the remainder can be distinguished by finite separators}
\label{sec_beta3}

For the tangle \comp\ it is true that every open set $\cO_{\modG}(X,\cC)$ gives rise to a clopen bipartition of the tangle space, namely
\begin{align*}
\big(\,\cO_{\modG}(X,\cC)\cap\Theta\,\big)&\oplus\big(\,\cO_{\modG}(X,\cC_X\setminus\cC)\cap\Theta\,\big),\\
\text{i.e. }\{\,\tau\in\Theta\mid\cC\in U(\tau,X)\,\}&\oplus\{\,\tau\in\Theta\mid\cC\notin U(\tau,X)\,\}.
\end{align*}
In fact, for every two distinct $\aleph_0$-tangles there exists such a clopen bipartition of the tangle space separating the two.
Our next target is to prove that any two components of the remainder of a graph are---just as the $\aleph_0$-tangles---distinguished by a finite order separation. 
That is why we start by studying a possible analogue $\cO_{\beta G}(X,\cC)$ of $\cO_{\modG}(X,\cC)$ for $\beta G$.

\begin{notation}
Given $X\in\cX$ and $\cC\subseteq\cC_X$ we write $G[X,\cC]$ for $G[X\cup V[\cC]]$.
If $\tau$ is an $\aleph_0$-tangle of $G$ and $\gamma$ is an element of $\Gamma$, then we write $G[\tau,\gamma]$ for $G[X(\gamma),\cC(\tau,\gamma)]$.
\end{notation}

For every $X\in\cX$ and $\cC\subset\cC_X$ we let 
\begin{align*}
\cO_{\beta G}(X,\cC):=\closureIn{G[X,\cC]}{\beta G}\setminus G[X]
\end{align*}
which is open in $\beta G$ as a consequence of $\beta G=\closure{G[X,\cC]}\cup \closure{G[X,\cC_X\setminus\cC]}$ and $\closure{G[X,\cC]}\cap \closure{G[X,\cC_X\setminus\cC]}=\closure{G[X]}=G[X]$ (see Theorem~\ref{thm_stonecechchar}~(v)).
Before we check that $\cO_{\beta G}(X,\cC)$ gives rise to clopen bipartitions of $G^\ast$ and $G^\times$, we prove a lemma:

\begin{lemma}
\label{SCdecomposition}
For all $X\in\cX$ and $\cC\subset\cC_X$ we have
\begin{align*}
\closureIn{G[X,\cC]}{\beta G}\subset O_{\beta G}[X]\sqcup\closureIn{\medcup\cC}{\beta G}.
\end{align*}
In particular, for all $\gamma\in\Gamma$ we have
\begin{align*}
\beta G=O_{\beta G} [X(\gamma)]\sqcup\bigsqcup_{\cC\in P(\gamma)}\closureIn{\medcup\cC}{\beta G}.
\end{align*}
\end{lemma}
\begin{proof}
Due to $\beta G=\bigcup_{\cC\in P(\gamma)}\closureIn{G[X(\gamma),\cC]}{\beta G}$ it suffices to show the first statement:
\begin{align*}
\closure{G[X,\cC]} &=G[X]\cup\bigcup_{x\in X}\closure{\medcup E(x,\medcup\cC)}\cup \closure{\medcup\cC}\\
&\subset\bigcup_{x\in X}\big( O_{\beta G} (x)\sqcup\closure{N(x)\cap\medcup\cC} \,\big)\cup\closure{\medcup\cC}=O_{\beta G}[X]\sqcup\closure{\medcup\cC}
\end{align*}
where at the ``$\subset$'' we used Theorem~\ref{thm_stonecechchar}~(v) for
\begin{align*}
\closure{\medcup E(x,\medcup\cC)} &=\p{ \closure{\medcup E(x,\medcup\cC)}\setminus\closure{N(x)} } \sqcup \p{ \closure{\medcup E(x,\medcup\cC)}\cap\closure{N(x)} }\\
&\subset O_{\beta G}(x)\sqcup \closure{(\medcup E(x,\medcup\cC))\cap N(x)}=O_{\beta G}(x)\sqcup \closure{N(x)\cap\medcup\cC}.\qedhere
\end{align*}
\end{proof}

\begin{lemmaDefinition}\label{Px}\label{finitePartitionOfGast}
Let any $(X,P)\in\Gamma$ be given. Then
\begin{enumerate}
\item $P_\ast:=\big\{\,\closureIn{G[X,\cC]}{\beta G}\cap G^\ast\;\big\vert\,\cC\in P\,\big\}$ and
\item $P_\times:=\big\{\,\closureIn{\bigcup\cC}{\beta G}\cap G^\times\;\big\vert\,\cC\in P\,\big\}$
\end{enumerate}
are finite separations of $G^\ast$ and $G^\times$ into clopen subsets.
\end{lemmaDefinition}
\begin{proof}
(i). First observe that 
\[\beta G = \closure{G} = \closure{\bigcup_{\cC \in P} G[X,\cC]} = \bigcup_{\cC \in P} \closure{G[X,\cC]}.\] 
At the same time, however, since every $G[X,\cC]$ is a subgraph, and hence a closed subset of $G$, for all $\cC\neq\cC'\in P$ it follows from Theorem~\ref{thm_stonecechchar}~(v) that
\begin{align*}
\closure{G[X,\cC]}\cap\closure{G[X,\cC']}=\closure{G[X,\cC]\cap G[X,\cC']}=\closure{G[X]}=G[X]\subset G
\end{align*}
where the last equality follows from the fact that compact subsets of Hausdorff spaces are closed. Hence, we see that $G^*$ is a disjoint union of finitely many closed sets $G^* = \bigsqcup_{\cC \in P} \big(\,\closure{G[X,\cC]} \cap G^*\big)$.

(ii) follows from (i) with Lemma~\ref{SCdecomposition}.
\end{proof}

\begin{notation}
We write ${\G*sim}$ and ${\Gxsim}$ for the equivalence relations on $G^\ast$ and $G^\times$ whose classes are precisely the connected components of $G^\ast$ and $G^\times$ respectively.
If $C$ is a component of $G^\times$ we write $\hat{C}$ for the unique component of $G^\ast$ including it.
\end{notation}

Our next lemma, the so-called \emph{Separating Lemma}, can be considered as our main technical result of this paper, yielding that distinct components of $G^*$ can be distinguished by a finite order separation of the graph $G$, see Corollaries~\ref{GxCompsInj} and \ref{finiteseparator} below. However, we state the lemma in a slightly more general form, so that we can also apply it in Section~\ref{sec6} when proving Theorem~\ref{mainresult2}. For this, we shall need the following notion of ``tame'':

\begin{definition}
We call a subset $A\subset\beta G$ \emph{tame} if it is ${\Gxsim}$-closed and for every component $C$ of $G^\ast$ meeting $A$ in a point of $O_{\beta G}[V]$ (cf.\ Def.~\ref{def_Gcross}) we have $C\subset A$.
\end{definition}

Here, for a set $M$ we say that $M$ is $R$-closed for $R$ an equivalence relation on any set $N$, if for all $m\in M$ and $n\in N$ with $mRn$ there holds $n\in M$ (phrased differently, whenever $M$ meets an $R$-class then it contains that class entirely, but $M$ might also contain points that do not lie in any $R$-class).

\begin{example}
All ${\G*sim}$-closed subsets of $\beta G$ and all ${\Gxsim}$-closed subsets of $G^\times$ are tame, but both $G^\ast\setminus G^\times$ and $O_{\beta G}[V]$ are not tame as soon as $G$ is not locally finite.
\end{example}

\begin{lemma}[\v{S}ura-Bura Lemma~{\cite[Theorem~6.1.23]{EngelkingBook}}]\label{surabura1}
If $C_1$ and $C_2$ are distinct components of a compact Hausdorff space $X$, there is a clopen bipartition $A\oplus B$ of $X$ with $C_1\subset A$ and $C_2\subset B$.
\end{lemma}

\begin{lemma}[Separating Lemma]
\label{SuperLemma}
Let $A,B\subset \beta G$ be two disjoint closed and tame subsets.
Then there is a finite $X\subset V(G)$ and a bipartition $\{\cC_1,\cC_2\}$ of $\cC_X$ with $A\subset\closureIn{G[X,\cC_1]}{\beta G}$ and $B\subset\closureIn{G[X,\cC_2]}{\beta G}$.
\end{lemma}
\begin{proof}
Given $A$ and $B$ we use normality of $G^\times$ and a compactness argument to deduce from Lemma~\ref{surabura1} that there is a clopen bipartition $K_A\oplus K_B$ of $G^\times$ with $A\cap G^\times\subset K_A$ and $B\cap G^\times\subset K_B$.
Put $A'=A\cup K_A$ and $B'=B\cup K_B$ so $A'$ and $B'$ are closed and disjoint subsets of $\beta G$. 
Using that $\beta G$ is normal we find disjoint open sets $O_A,O_B\subset\beta G$ with $A'\subset O_A$ and $B'\subset O_B$.
Next, since
\begin{align*}
\bigcap_{v\in V}(\beta G\setminus O_{\beta G}(v))=G^\times=K_A\oplus K_B\subset O_A\sqcup O_B
\end{align*}
is an intersection of closed sets which is contained in the open set $O_A\sqcup O_B$, it follows from compactness that there are finitely many vertices $v_1,\ldots,v_n$ such that
\begin{align*}
\bigcap_{i=1}^n (\beta G\setminus O_{\beta G}(v_i))\subset O_A\sqcup O_B.
\end{align*}
Put $\Xi=\{v_1,\ldots,v_n\}$. Then $O_A\sqcup O_B$ induces a clopen bipartition $K_A'\oplus K_B'$ of the closed subspace $\beta G\setminus O_{\beta G}[\Xi]$ of $\beta G$ which in turn induces a bipartition $Q=\{\cA,\cB\}$ of $\cC_\Xi$ via
\begin{align*}
\cA=\{\,C\in\cC_\Xi\mid C\subset K_A'\,\}\quad\text{and}\quad\cB=\{\,C\in\cC_\Xi\mid C\subset K_B'\,\}.
\end{align*}
In particular, we have 
\begin{align}\label{randomEq}
\closureIn{\medcup\cA}{\beta G}\subset K_A'\quad\text{and}\quad\closureIn{\medcup\cB}{\beta G}\subset K_B'.
\end{align}
Moreover, by Lemma~\ref{Px}, $Q_\times$ must be the clopen bipartition $K_A\oplus K_B$ of $G^\times$. 

Now we want that
\begin{align}\label{wanted}
A\subset\closureIn{G[\Xi,\cA]}{\beta G}\quad\text{and}\quad B\subset\closureIn{G[\Xi,\cB]}{\beta G},
\end{align}
but with the help of Lemma~\ref{SCdecomposition} and (\ref{randomEq}) we only get
\begin{align*}
&A\subset \beta G\setminus\closureIn{\medcup\cB}{\beta G}=\closureIn{G[\Xi,\cA]}{\beta G}\cup O_{\beta G}[\Xi]\\
\text{and}\quad &B\subset\beta G\setminus\closureIn{\medcup\cA}{\beta G}=\closureIn{G[\Xi,\cB]}{\beta G}\cup O_{\beta G}[\Xi]
\end{align*}
with $A$ and $B$ possibly meeting $O_{\beta G}[\Xi]$.
To resolve this issue, we will find a way to widen $\Xi$ by adding only finitely many vertices, and adjusting $\cA$ and $\cB$ accordingly so as to make (\ref{wanted}) true.

For this, we note first that 
\begin{align}\label{SuperLemmaEq1}
A\cap G^\ast\subset \closureIn{G[\Xi,\cA]}{\beta G}.
\end{align}
Indeed, we know that $A$ is tame, that each component of $G^\ast$ meets $V^\ast\subset G^\times$ (see Lemmas~\ref{lem_componentsandgraphs} and~\ref{infiniteCmeetsGtimes}), and that
\begin{align*}
A\cap G^\times\subset K_A=\closureIn{\medcup\cA}{\beta G}\cap G^\times\subset\closureIn{G[\Xi,\cA]}{\beta G}
\end{align*}
where $\closureIn{G[\Xi,\cA]}{\beta G}\cap G^\ast$ is clopen (Lemma~\ref{Px}); combining these facts yields (\ref{SuperLemmaEq1}).

Second, we show that there exists a finite set $F_A$ of edges of $G$ with 
\begin{align}\label{SuperLemmaEq2}
A\cap G\subset G[\Xi,\cA]\cup\medcup F_A.
\end{align}
Indeed, 
by $A\subset\closureIn{G[\Xi,\cA]}{\beta G}\cup O_{\beta G}[\Xi]$,
it suffices to show that
\begin{align*}
F_A:=\{\,e\in E(X,\medcup\cB)\mid \mathring{e}\text{ meets }A\,\}
\end{align*}
is finite.
Suppose for a contradiction that $F_A$ is infinite, and for every edge $e\in F_A$ pick some $i_e\in\mathring{e}\cap A$.
Then $\closureIn{\{\,i_e\mid e\in F_A\,\}}{\beta G}\subset A$ meets $G^*\cap \closureIn{G[X,\cB]}{\beta G}\cap O_{\beta G}[\Xi]$ in some component $C$ of $G^\ast$. 
But we noted earlier that each component of $G^\ast$ meets $V^\ast\subset G^\times$, so the tame set $A$ meeting $C$ means that $\emptyset\neq C\cap G^\times\subset A\cap K_B$, a contradiction.
Of course, corresponding versions of (\ref{SuperLemmaEq1}) and (\ref{SuperLemmaEq2}) hold for $B$.

Finally, we use (\ref{SuperLemmaEq1}) and (\ref{SuperLemmaEq2}) to yield a true version of (\ref{wanted}).
For this, we let $X$ be the finite vertex set obtained from $\Xi$ by adding the endvertices of the edges in $F_A\cup F_B$, and we put $\cC_1=\cb_{X,\Xi}^{-1}(\cA)$ and $\cC_2=\cb_{X,\Xi}^{-1}(\cB)$.
Due to $G[X,\cC_1]\supset G[\Xi,\cA]\cup\bigcup F_A$ and $G[X,\cC_2]\supset G[\Xi,\cB]\cup\bigcup F_B$, we may use (\ref{SuperLemmaEq1}) and (\ref{SuperLemmaEq2}) to deduce that
\begin{align*}
A\subset\closureIn{G[X,\cC_1]}{\beta G}\quad\text{and}\quad B\subset \closureIn{G[X,\cC_2]}{\beta G}.&\qedhere
\end{align*}
\end{proof}

Using Lemma~\ref{SCdecomposition} we obtain the following corollary:

\begin{corollary}\label{GxCompsInj}
For every pair of distinct components $C_1,C_2$ of $G^\times$ there is a finite $X\subset V(G)$ and a bipartition $P=\{\cC_1,\cC_2\}$ of $\cC_X$ such that the components $\hat{C}_1\supset C_1$ and $\hat{C}_2\supset C_2$ of $G^\ast$ are separated by the clopen bipartition $P_\ast$ of $G^\ast$.\qed
\end{corollary}

\begin{lemma}\label{bijectiveComponentExtension}
The map $C\mapsto\hat{C}$ defines a bijection between $G^\times/{\Gxsim}$ and $G^\ast/{\G*sim}$.
\end{lemma}
\begin{proof}
Each component of $G^\ast$ meets $V^\ast\subset G^\times$ (see Lemmas~\ref{lem_componentsandgraphs} and~\ref{infiniteCmeetsGtimes}), so the map $C \mapsto \hat{C}$ is onto. It is injective by Corollary~\ref{GxCompsInj}.
\end{proof}

Corollary~\ref{GxCompsInj} and Lemma~\ref{bijectiveComponentExtension} yield another important result:

\begin{corollary}\label{finiteseparator}
For every pair of distinct components $C_1,C_2$ of $G^\ast$ there is a finite $X\subset V(G)$ and a bipartition $P$ of $\cC_X$ such that the clopen bipartition $P_\ast$ of $G^\ast$ separates $C_1$ and $C_2$.\qed 
\end{corollary}

\begin{corollary}
\label{distinguishComponents}
The quotients $G^\times/{\Gxsim}$ and $G^\ast/{\G*sim}$ are \HD . \qed
\end{corollary}

\begin{theorem}
\label{GtimesSimCongGastSim}
For any graph $G$, we have $G^\times/{\Gxsim}\cong G^\ast/{\G*sim}$.
\end{theorem}
\begin{proof}
Let $\hat{\iota}\colon G^\times/{\Gxsim}\to G^\ast/{\G*sim}$ map $C$ to $\hat{C}$. By Lemma~\ref{bijectiveComponentExtension} this is a bijection.
Denote the quotient map $G^\ast\to G^\ast/{\G*sim}$ by $q_\ast$.
Clearly, the diagram
\begin{center}
\begin{tikzcd}
G^\times\arrow[r, right hook->, "\iota"]\arrow[d, two heads] & G^\ast\arrow[d, "q_\ast", two heads] \\
G^\times/{\Gxsim}\arrow[r, "\hat{\iota}"] & G^\ast/{\G*sim}
\end{tikzcd}
\end{center}
commutes. Since $G^\times/{\Gxsim}$ is compact and $G^\ast/{\G*sim}$ is \HD\ (Corollary~\ref{distinguishComponents}), to show that $\hat{\iota}$ is a homeomorphism it suffices to verify continuity. But note that by the quotient topology, $\hat{\iota}$ is continuous if and only if $q_\ast \circ \iota $ is continuous.
\end{proof}

\subsection{Comparing \texorpdfstring{$\cP$}{P} with \texorpdfstring{$G^\times$}{Gx}}

Now that we are able to distinguish distinct components of the remainder by some $\gamma\in\Gamma$, the next step is to use this to show $\Theta\cong G^\ast/{\G*sim}$.
Technically, we will achieve this by showing $\cP\cong G^\times/{\Gxsim}$ instead.

For every $\gamma\in\Gamma$ let $\sigma_\gamma\colon G^\times\to \cP_\gamma$ map every point $x\in G^\times$ to the $\cC\in \cP_\gamma$ whose induced clopen partition class $\closureIn{\bigcup\cC}{\beta G}\cap G^\times\in P(\gamma)_\times$ containing $x$, i.e. including the connected component of $G^\times$ containing $x$.

\begin{lemma}
\label{sigmaGammaContinuousSurjections}
The maps $\sigma_\gamma$ are continuous surjections.
\end{lemma}
\begin{proof}
To see that $\sigma_\gamma$ is continuous, observe that
\begin{align*}
\sigma_\gamma^{-1}(\cC)=\closureIn{\medcup\cC}{\beta G}\cap G^\times\in P(\gamma)_\times
\end{align*}
and recall that partition classes of $P(\gamma)_\times$ are clopen in $G^\times$.

The map $\sigma_\gamma$ is surjective: since every $\cC\in \cP_\gamma$ is such that $V[\cC]$ is infinite, Lemma~\ref{infiniteCmeetsGtimes} ensures that $\closureIn{\bigcup\cC}{\beta G}\cap G^\times$ is non-empty.
\end{proof}

\begin{lemma}
\label{sigmaGammeCompatible}
The maps $\sigma_\gamma$ are compatible.
\end{lemma}
\begin{proof}
For this assertion it suffices to show that whenever $(X,P)\le (X',P')$, then we have $P_\times \preceq P'_\times$, i.e.\ the finite clopen partition $P'_\times$ refines that partition of $G^\times$ induced by $P_\times$.
To see this, consider any $\cC' \in P'$. Since $P'$ refines $P\rise X'$, there is a unique $\cC \in P$ with $\cC'\rest X \subset \cC$.
Thus $\closureIn{\bigcup\cC'}{\beta G}\cap G^\times\subset\closureIn{\bigcup\cC}{\beta G}\cap G^\times$ follows.
\end{proof}

We put $\sigma=\invLim \sigma_\gamma\colon G^\times\to\cP$, and we aim to show that $\sigma$ gives rise to a homeomorphism between $G^\times/{\Gxsim}$ and $\cP$.

\begin{lemma}
\label{sigmaCTSsurjection}
The map $\sigma\colon G^\times\to\cP$ is a continuous surjection.
\end{lemma}
\begin{proof}
We combine Lemmas~\ref{sigmaGammaContinuousSurjections} and~\ref{sigmaGammeCompatible} with the fact that compatible continuous surjections from a compact space onto \HD\ spaces combine into one continuous surjection onto the inverse limit of their image spaces (cf.~\cite[Corollary~3.2.16]{EngelkingBook}).
\end{proof}

\begin{lemma}
\label{sigmaFibresConnected}
The fibres of $\sigma$ are precisely the connected components of $G^\times$.
\end{lemma}
\begin{proof}
First, it is clear by the definition of the $\sigma_\gamma$ that every $\sigma_\gamma$ is constant on connected components of $G^\times$.
Conversely, we need to argue that for any pair of distinct components $C_1$ and $C_2$ of $G^\times$ there is some $\sigma_\gamma$ with $\sigma_\gamma \restriction C_1 \neq \sigma_\gamma \restriction C_1$.
Such a $\sigma_\gamma$ is provided by Corollary~\ref{GxCompsInj}.
\end{proof}

\begin{proposition}
\label{GtimesSimCongP}
$G^\times/{\Gxsim}\cong\cP$.
\end{proposition}
\begin{proof}
It is well-known that every continuous surjection $f\colon X \twoheadrightarrow Y$ from a compact space $X$ onto a Hausdorff space $Y$ gives rise to a homeomorphism between the quotient $X  / \{\,f^{-1}(y)\mid y \in Y\,\}$ over the fibres of $f$, and the space $Y$.
Thus, it follows from Lemmas~\ref{sigmaCTSsurjection} and~\ref{sigmaFibresConnected}, that
\[G^\times/{\Gxsim} = G^\times /\{\,\sigma^{-1}(\xi)\mid\xi \in \cP\,\} \cong\cP. \qedhere \]
\end{proof}

We now have all ingredients to prove the following key result that is essential for the proof our main theorem:

\begin{theorem}
\label{mainresult1}
The tangle space $\Theta$ of any graph $G$ is homeomorphic to the quotient $G^* /{\G*sim}$ of the Stone-\v{C}ech remainder $G^*$ of $G$, where each connected component of $G^*$ is collapsed to a single point. 
\end{theorem}
\begin{proof}[Proof of Theorem~\ref{mainresult1}]
\label{proofthm}
Theorem~\ref{GtimesSimCongGastSim}, Proposition~\ref{GtimesSimCongP} and Theorem~\ref{ThetaInvLim} yield
\[G^\ast/{\G*sim}\cong G^\times/{\Gxsim}  \cong \cP \cong \Theta. \qedhere \]
\end{proof}

We write $\tau_\ast$ for the component of $G^\ast$ corresponding to $\tau$ and $\tau_\times$ for the component $\tau_\ast\cap G^\times$ of $G^\times$ corresponding to $\tau$ (cf. Theorem~\ref{mainresult1} and Lemma~\ref{bijectiveComponentExtension}).

\begin{theorem}\label{tangleComponents}\label{tauTimesComponent}
If $\tau$ is an $\aleph_0$-tangle of $G$, then
\begin{enumerate}
\item $\tau_\ast = G^\ast\cap\bigcap_{\gamma\in\Gamma}\closureIn{ G[\tau,\gamma] }{\beta G}$ and
\item $\tau_\times =\bigcap_{\gamma\in\Gamma}\closureIn{\bigcup\cC(\tau,\gamma)}{\beta G}=G^\times\cap\bigcap_{\gamma\in\Gamma}\closureIn{G[\tau,\gamma]}{\beta G}=\tau_\ast\cap G^\times$
\end{enumerate}
are the components of $G^\ast$ and $G^\times$ corresponding to $\tau$ respectively.
\end{theorem}
In statement (i) of the theorem, the intersection with $G^\ast$ is really necessary---we will see the reason for this in Proposition~\ref{componentGraphInteraction}.
\begin{proof}[Proof of Theorem~\ref{tangleComponents}]
We show (ii) first. The first equality is evident from the definition of $\sigma$, and the centre equality follows from Lemma~\ref{SCdecomposition} with 
\begin{align*}
G^\times=\bigcap_{\gamma\in\Gamma}\big(\beta G\setminus O_{\beta G}[X(\gamma)]\big).
\end{align*}

(i). By Corollary~\ref{finiteseparator}, the right-hand side contains at most one connected component of $G^\ast$. We have $\tau_\times\subset\closure{G[\tau,\gamma]}$ for all $\gamma\in\Gamma$ by (ii), so $\tau_\ast=\hat{\tau}_\times\subset\closure{G[\tau,\gamma]}$ holds for all $\gamma$ as well (see Lemma~\ref{Px}), finishing the proof.
\end{proof}

\section{Obtaining the tangle compactification\texorpdfstring{\\}{}from the Stone-\texorpdfstring{\v{C}}{C}ech compactification}
\label{sec6}

Now that we know $\Theta\cong G^\ast/{\G*sim}$, our next target is the proof of our main result, Theorem~\ref{mainresult2}. 
For this, recall that $\cO_{\beta G}(X,\cC)=\closureIn{G[X,\cC]}{\beta G}\setminus G[X]$,
and that Lemma~\ref{finitePartitionOfGast} and Theorem~\ref{tangleComponents} ensure that $\cO_{\beta G}(X,\cC)$ is ${\G*sim}$-closed and includes precisely the components $\tau_\ast$ of $G^\ast$ with $\cC\in U(\tau,X)$.

\begin{lemma}
\label{magicLemma}
Let $A\subset\beta G$ be closed and ${\G*sim}$-closed, and let $\tau$ be an $\aleph_0$-tangle of $G$.
If $A$ avoids $\tau_\ast$, then there are $X\in\cX$ and $\cC\subset\cC_X$ with $\tau_\ast\subset\cO_{\beta G}(X,\cC)\subset \beta G\setminus A$.
\end{lemma}
\begin{proof}
By the Separating Lemma~\ref{SuperLemma} there is $X\in\cX$ and a bipartition $\{\cC_1,\cC_2\}$ of $\cC_X$ with $A\subset\closure{G[X,\cC_1]}$ and $\closure{\tau_\ast}\subset\closure{G[X,\cC_2]}$.
Then $\tau_\ast\subset\cO_{\beta G}(X,\cC_2)\subset \beta G\setminus A$.
\end{proof}

We write $\widehat{\beta G}$ for the topological space obtained from $\beta G$ by declaring $G$ to be open, and we write $\hat{G}$ for the quotient $\widehat{\beta G}/{\G*sim}$.
Since $\beta G$ contains $G$ as a subspace, all the open sets of $G$ are open in $\widehat{\beta G}$ as well; and since ${\G*sim}$ does not affect $G$, all the open sets of $G$ are also open in $\hat{G}$.
As a consequence, the open sets of $\beta G$ plus the open sets of $G$ form a basis for the topology of $\widehat{\beta G}$, yielding that

\begin{lemma}\label{GhatBasis}
The open sets of $(\beta G)/{\G*sim}$ plus the open sets of $G$ form a basis for the topology of $\hat{G}$.\qed
\end{lemma}

We define a bijection $\Psi\colon\hat{G}\to\modG$ by letting it be the identity on $G$ and letting it send each ${\G*sim}$-class $\tau_\ast$ to its corresponding $\aleph_0$-tangle $\tau$.

\begin{lemma}
\label{PsiCts}
The map $\Psi$ is continuous.
\end{lemma}
\begin{proof}
Since the open sets of $G$ are open in both $\modG$ and $\hat{G}$, it suffices to show that the preimage of any $\cO_{\modG}(X,\cC)$ is open in $\hat{G}$, and it is:
\begin{align*}
\Psi^{-1}\big(\cO_{\modG}(X,\cC)\big)=\cO_{\beta G}(X,\cC)/{\G*sim}.&\qedhere
\end{align*}
\end{proof}

\begin{lemma}
\label{PsiClosed}
The map $\Psi$ is closed.
\end{lemma}
\begin{proof}
Let $A$ be any closed subset of $\hat{G}$; we show that $\Psi[A]$ is closed in $\modG$.
For this, let $\xi$ be any point of $\modG\setminus\Psi[A]$, and let $\cB$ be the basis for the topology of $\hat{G}$ provided by Lemma~\ref{GhatBasis}.

If $\xi$ is a point of $G$, then we find an open neighbourhood $O$ of $\xi$ in $G$ avoiding $A$ since $A$ is closed in $\hat{G}$.
Then $O$ witnesses $\xi\notin\closure{\Psi[A]}$ as well.

Otherwise $\xi$ is an $\aleph_0$-tangle $\tau\in\Theta\setminus\Psi[A]$.
The set $A$ is closed in $\hat{G}$, but it need not be closed in $(\beta G)/{\G*sim}$. 
Let us consider the closure $B$ of $A$ in $(\beta G)/{\G*sim}$ and show $B\setminus A\subset G$ (actually, one can even show that $B$ adds only some vertices of infinite degree to $A$, but $B\setminus A\subset G$ suffices for our cause).
Each point of $\hat{G}\setminus G$ that is not contained in $A$ has an open neighbourhood from the basis $\cB$ avoiding $A$. 
Since all these neighbourhoods are not included in $G$, they must be open sets of $(\beta G)/{\G*sim}$, yielding $B\setminus A\subset G$.
Therefore, the closed set $B'=\bigcup B$ of $\beta G$ avoids the component $\tau_\ast$ of $G^\ast$ corresponding to $\tau$, and since $B'$ is also ${\G*sim}$-closed our Lemma~\ref{magicLemma} 
yields $X\in\cX$ and $\cC\subset\cC_X$ such that $\tau_\ast\subset\cO_{\beta G}(X,\cC)\subset\beta G\setminus B'$.
Therefore, the open neighbourhood $\cO_{\modG}(X,\cC)$ of $\tau$ avoids $\Psi[A]$.
\end{proof}

\begin{customthm}{\ref{mainresult2}}
The tangle compactification $\modG$ of any graph $G$ is homeomorphic to the quotient $(\beta G, \tau') /{\G*sim}$ where $\tau'$ is the finer topology on $\beta G$ obtained from $\beta G$ by declaring $G$ to be open in $\beta G$ and then collapsing each connected component of $G^*$ to a single point. 
\end{customthm}
\begin{proof}
Lemma~\ref{PsiCts} and Lemma~\ref{PsiClosed} yield a homeomorphism.
\end{proof}

\section{Three observations about the Stone-\texorpdfstring{\v{C}}{C}ech compactification}
\label{sec7}

Given $\M_E$ and an ultrafilter $U\in \beta E$ we write $P_U$ for the collection of all points of $\uI_U$ that are of the form $x_U$ for some family $(\,x_e\mid e\in E\,)$ of points $x_e\in\uI_e$.
By \cite[Proposition~2.6]{hart}, the set $P_U\setminus\{0_U,1_U\}$ is dense in $\uI_U$.

\begin{theorem}
If $G$ is an infinite graph that is not locally finite, then no \comp\ of $G$ can both be \HD\ and have a totally disconnected remainder.
\end{theorem}
\begin{proof}
Suppose for a contradiction that $\alpha G$ is a \HDcomp\ of $G$ with totally disconnected remainder, and let $v$ be a vertex of $G$ of infinite degree.
Consider a representation $\M_E / {\Gsim}$ of $G$, so Theorem~\ref{lem_stonecechquotient} yields $\beta G=(\beta\M_E)/{\bGsim}$ and we find a free ultrafilter $U\in E^\ast$ with $[0_U]_{\bGsim}=v$, say.
The set $P_U\setminus\{0_U,1_U\}$ is dense in $\uI_U$, so every open neighbourhood of $v$ in $\beta G$ meets $P_U\setminus\{0_U,1_U\}$.
In order to use this to derive a contradiction, we need to know more about $\alpha G$ first.

The \HDcomp\ $\alpha G$ can be obtained from $\beta G$ as a quotient $\beta G/{\approx}$ where ${\approx}$ is an equivalence relation on $G^\ast$.
Since $\alpha G$ has a totally disconnected remainder and since the (continuous) restriction of the quotient map to components of $G^\ast$ preserves connectedness, we deduce that the equivalence relation ${\approx}$ must refine ${\G*sim}$.
Consequently, the connected subspace $\check{\uI}_U$ of $G^\ast$ (cf.~Corollary~\ref{openIntervalConnected}) is included in a single ${\approx}$-class $x$, say.
To yield a contradiction, it suffices to show that every open neighbourhood $O$ of $v$ in $\alpha G$ contains $x$.
And indeed: if we view $\alpha G$ as the quotient $(\beta G)/{\approx}$ of $\beta G$, then $\bigcup O$ is open in $\beta G$ and ${\approx}$-closed. 
Using that $\bigcup O$ meets $P_U\setminus\{0_U,1_U\}$ and ${\approx}$ refines ${\G*sim}$ we deduce that $x\subset\bigcup O$, i.e. $x\in O$.
\end{proof}

For our the second observation we need a short lemma and some notation:
Since $G$ is dense in $\beta G$, so is the locally compact subspace formed by the inner edge points and the vertices of finite degree, and hence \cite[Theorem~3.3.9]{EngelkingBook} yields:

\begin{lemma}\label{innerEdgesOpen}
If $G$ is a graph, then $\mathring{E}\subset G$ is open in $\beta G$.\qed
\end{lemma}

Given an end $\omega$ of $G$ we write $\Delta \omega$ for the set of those vertices dominating it.
Our second observation describes explicitly how the connected components of the \SC\ remainder of $G$ interact with $G$.

\begin{proposition}\label{componentGraphInteraction}
Let $G$ be any graph, and let $\M_E/{\Gsim}$ be a representation of $G$.
\begin{enumerate}
\item If $\tau$ is an ultrafilter tangle of $G$, then $\closureIn{ \tau_\ast }{\beta G}=\tau_\ast\sqcup X_\tau$, and for each $x\in X_\tau$ there is an ultrafilter $U\in E^\ast$ with $[0_U]_{\bGsim}= x$, say, and with $\check{\uI}_U\subset \tau_\ast$.
\item If $\omega$ is an end of $G$, then $\closureIn{ \omega_\ast }{\beta G}=\omega_\ast\sqcup\Delta\omega$, and for each $x\in \Delta\omega$ there is an ultrafilter $U\in E^\ast$ with $[0_U]_{\bGsim} =x$, say, and with $\check{\uI}_U\subset \omega_\ast$.
\end{enumerate}
\end{proposition}
\begin{proof}
(i). First, we show that $\closureIn{\tau_\ast}{\beta G}$ avoids $G\setminus X_\tau$ (where $G$ is the 1-complex).
Since $\mathring{E}$ is open in $\beta G$ (Lemma~\ref{innerEdgesOpen}) we may assume that $\closureIn{\tau_\ast}{\beta G}\cap G\subset V$.
Let $v$ be any vertex of $G$ that is not in $X_\tau$, and let $C$ be the (graph) component of $G-X_\tau$ with $v\in C$. Then $v\notin\closure{ G[X_\tau,\cC_{X_\tau}\setminus \{C\}]}$ implies $v\notin\closureIn{\tau_\ast}{\beta G}$ by Theorem~\ref{tangleComponents} as desired.
Therefore, $\closureIn{\tau_\ast}{\beta G}\cap G\subset X_\tau$.

Now suppose that any vertex $x\in X_\tau$ is given.
Write $\Gamma_x$ for the set of all $\gamma\in\Gamma$ with $x\in X(\gamma)$, and given $\gamma\in\Gamma_x$ put $F_\gamma=E(x,\bigcup\cC(\tau,\gamma))$.
The sets $F_\gamma$ are infinite due to~\cite[Lemma~4.4]{EndsTanglesCrit}.
We consider the filter on $E(x)$ that is given by the up-closure of the collection
$\{\,F_\gamma\mid \gamma\in\Gamma_x\,\}\subset 2^{E(x)}$
(from the directedness of $\Gamma_x$ it follows that this collection is directed by reverse inclusion, which is enough to ensure that we get a filter).
Next, we extend this filter to an ultrafilter $U$ on $E(G)$, and note that $U$ must be free.
Due to $E(x)\in U$ we may assume without loss of generality that there is some $F\in U$ with $F\subset E(x)$ and $\{\,0_e\mid e\in F\,\}\subset x$ where we view $x$ as a ${\Gsim}$-class of $\M_E$.
Then $0_U\in\closureIn{\{\,0_e\mid e\in F\,\}}{\beta\M_E}$ implies $[0_U]_{\bGsim}=x$ as a consequence of $\beta G=(\beta\M_E)/{\bGsim}$, Theorem~\ref{lem_stonecechquotient}.
If we can show that $\check{\uI}_U$ is included in $\closureIn{ G[\tau,\gamma] }{\beta G}$ for all $\gamma\in\Gamma_x$, then we are done since $\Gamma_x$ is cofinal in $\Gamma$ and $\tau_\ast$ can be written as the directed intersection $G^\ast\cap\bigcap_{\gamma\in\Gamma}\closure{ G[\tau,\gamma] }$ (cf.~Theorem~\ref{tangleComponents}).
For this, let any $\gamma\in\Gamma_x$ be given.
Since $\check{\uI}_U\subset G^\ast$ is connected (cf.~Corollary~\ref{openIntervalConnected}) and $G^\ast\cap\closure{ G[\tau,\gamma] }$ is clopen in $G^\ast$ (cf.~Lemma~\ref{finitePartitionOfGast}), it suffices to show that $\check{\uI}_U$ meets $\closure{ G[\tau,\gamma] }$ in $(\frac{1}{2})_U$.
And indeed we have
\begin{align*}
(\tfrac{1}{2})_U\in\closureIn{\{\,(\tfrac{1}{2})_e\mid e\in F_\gamma\,\}}{\beta\M_E}
\end{align*}
which implies $(\frac{1}{2})_U\in\closure{ G[\tau,\gamma] }$ as desired.

(ii). This is proved similar to (i), where to show $\closureIn{\omega_\ast}{\beta G}\cap G\subseteq\Delta\omega$ we use that for every vertex $v$ of $G$ not dominating $\omega$ there is $X\in \cX$ separating $v$ from $C(X,\omega)$ in that $v\notin X\cup C(X,\omega)$ so in particular $v\notin\closureIn{ G[X,\{C(X,\omega)\}]}{\beta G}\supset\closureIn{\omega_\ast}{\beta G}$.
\end{proof}

For the study of locally finite connected graphs, the so-called \emph{Jumping Arc Lemma} (cf.~\cite[Lemma~8.5.3]{Bible}) plays an important role.
By considering subcontinua of the \SC\ \comp\ instead of arcs in the Freudenthal \comp , we obtain the following quite strong generalisation of this lemma:

\begin{lemma}[Jumping `Arc' Lemma for the \SC\ \comp ]\ \\
Let $F\subset E$ be a cut of $G$ with sides $V_1,V_2$.
\begin{enumerate}
\item If $F$ is finite, then $\closure{G[V_1]}\oplus\closure{G[V_2]}$ is a clopen bipartition of $(\beta G)\setminus\mathring{F}$, and there is no subcontinuum of $(\beta G)\setminus\mathring{F}$ meeting both $V_1$ and $V_2$.
\item If $F$ is infinite, then $(\beta G)\setminus\mathring{F}$ might contain a subcontinuum meeting both $V_1$ and $V_2$.
This is the case, e.g., if both $G[V_1]$ and $G[V_2]$ are connected.
\end{enumerate}
Moreover, two vertices of $G$ lie in the same component (subcontinuum) of $(\beta G)\setminus\mathring{E}$ if and only if they lie on the same side of every finite cut of the graph $G$.
\end{lemma}
\begin{proof}
(i) is immediate from Theorem~\ref{thm_stonecechchar}~(v). For (ii), note that if both $G[V_1]$ and $G[V_2]$ are connected then $F$ is a bond so $(\beta G)\setminus\mathring{F}'=\subclosureIn{G\setminus\mathring{F}'}{\beta G}$ is a continuum for every finite $F'\subset F$ by Lemmas~\ref{closureconnd},~\ref{innerEdgesOpen} and Theorem~\ref{thm_stonecechchar}~(v). 
Hence $(\beta G)\setminus \mathring{F}$ is also a continuum as directed intersection of the continua $(\beta G)\setminus \mathring{F}'$, see Lemma~\ref{continuaintersection}.

Finally, note that, by (i), for the `moreover' part it suffices to show the backward direction.
For this, find infinitely many edge-disjoint paths $P_0,P_1,\hdots$ between the two vertices inductively, and note that by Lemmas~\ref{closureconnd},~\ref{continuaintersection} and~\ref{innerEdgesOpen} the intersection
\begin{align*}
\bigcap_{n\in\N}\closureIn{ \bigcup_{m>n} P_m }{\beta G}\subset (\beta G)\setminus\mathring{E}
\end{align*}
is a continuum containing the two vertices as desired.
\end{proof}


\bibliographystyle{amsplain}
\bibliography{TanglesSC}

\providecommand{\bysame}{\leavevmode\hbox to3em{\hrulefill}\thinspace}
\providecommand{\MR}{\relax\ifhmode\unskip\space\fi MR }
\providecommand{\MRhref}[2]{%
  \href{http://www.ams.org/mathscinet-getitem?mr=#1}{#2}
}
\providecommand{\href}[2]{#2}
\begin{thebibliography}{10}

\bibitem{aarts1993dimension}
J.M. Aarts and T.~Nishiura, \emph{Dimension and extensions}, vol.~48, Elsevier,
  1993.

\bibitem{Woess93}
D.I. Cartwright, P.M. Soardi, and W.~Woess, \emph{Martin and end
  compactifications for non-locally-finite graphs}, Trans.\ Am.\ Math.\ Soc.
  \textbf{338} (1993), 679--693.

\bibitem{Bible}
R.~Diestel, \emph{{Graph Theory}}, 4th ed., Springer, 2010.

\bibitem{EndsAndTangles}
\bysame, \emph{{Ends and Tangles}}, {Abh.\ Math.\ Sem.\ Univ.\ Hamburg}
  \textbf{87} (2017), no.~2, 223--244, {Special issue in memory of Rudolf
  Halin}, \href{https://arxiv.org/abs/1510.04050}{arXiv:1510.04050v3}.

\bibitem{AbstractSepSys}
\bysame, \emph{{Abstract Separation Systems}}, Order \textbf{35} (2018),
  157--170.

\bibitem{Ends}
R.~Diestel and D.~Kühn, \emph{{Graph-theoretical versus topological ends of
  graphs}}, J. Combin.\ Theory (Series B) \textbf{87}
  (\href{https://www.math.uni-hamburg.de/home/diestel/papers/TopEnds.pdf}{2003}),
  197--206.

\bibitem{EngelkingBook}
R.~Engelking, \emph{{General Topology}}, second ed., Sigma Series in Pure
  Mathematics, vol.~6, Heldermann Verlag, Berlin, 1989.

\bibitem{GillmanJerison}
L.~Gillman and M.~Jerison, \emph{{Rings of Continuous Functions}},
  Springer-Verlag, New York, 1976.

\bibitem{halin64}
R.~Halin, \emph{{Über unendliche Wege in Graphen}}, Math.\ Annalen
  \textbf{157} (1964), 125--137.

\bibitem{hart}
K.P. Hart, \emph{The \v{C}ech-stone compactification of the real line}, Recent
  Progress in General Topology, North-Holland, Amsterdam (1992), 317--352.

\bibitem{EndsTanglesCrit}
J.~Kurkofka and M.~Pitz, \emph{Ends, tangles and critical vertex sets}, Math.\
  Nachr. (2019), to appear,
  \href{https://arxiv.org/abs/1804.00588}{arXiv:1804.00588v1}.

\bibitem{Polat90}
N.~Polat, \emph{Topological aspects of infinite graphs}, Cycles and rays: basic
  structures in finite and infinite graphs, Proc. NATO Adv. Res. Workshop,
  Montreal/Can. 1987, NATO ASI Ser. (1990), 197--220.

\bibitem{GMX}
N.~Robertson and P.D. Seymour, \emph{{Graph Minors. X. Obstructions to
  Tree-Decomposition}}, J. Combin.\ Theory (Series B) \textbf{52} (1991),
  153--190.

\bibitem{MTeegen}
J.M. Teegen, \emph{{Abstract Tangles as an Inverse Limit, and a Tangle
  Compactification for Topological Spaces}}, Master's thesis, {Universität
  Hamburg}, 2017.

\end{thebibliography}

\end{document}